\documentclass[10pt]{amsart}
\usepackage{amssymb, amsrefs, mathrsfs, graphicx, subfig}
\usepackage[german, english]{babel}
\copyrightinfo{2010}{Matthias Keller, Daniel Lenz, Rados{\l}aw K. Wojciechowski}

\newtheorem{theorem}{Theorem}
\newtheorem{lemma}{Lemma}[section]
\newtheorem{thm}[lemma]{Theorem}
\newtheorem{proposition}[lemma]{Proposition}
\newtheorem{corollary}[lemma]{Corollary}

\theoremstyle{definition}
\newtheorem{definition}[lemma]{Definition}
\newtheorem*{example}{Example}

\theoremstyle{remark}
\newtheorem*{remark}{Remark}
\newtheorem*{remarks}{Remarks}

\numberwithin{equation}{section}

\newcommand{\Lf}{\widetilde{L}}

\newcommand{\Ln}{\widehat{\Delta}}
\newcommand{\tu}{\textup}

\newcommand{\tpdt}{\tfrac{d}{dt}}

\newcommand{\R}{{\mathbb R}}
\newcommand{\C}{{\mathbb C}}
\newcommand{\Q}{{\mathbb Q}}
\newcommand{\N}{{\mathbb N}}

\newcommand{\A}{{\widetilde{A}}}
\newcommand{\Lp}{{\Delta}}

\newcommand{\se}{\sigma_{\!\mathrm{ess}}}

\newcommand{\sd}{\sigma_{\!\mathrm{disc}}}
\newcommand{\lmess}{\lambda_0^{\tu{ess}}}

\newcommand{\al}{{\alpha}}
\newcommand{\be}{{\beta}}
\newcommand{\de}{{\delta}}

\newcommand{\ka}{{\kappa}}

\newcommand{\lm}{{\lambda}}

\newcommand{\si}{{\sigma}}

\newcommand{\as}[1]{\left\langle #1\right\rangle}

\newcommand{\ov}[1]{\overline{ #1}}
\newcommand{\ow}[1]{\widetilde{ #1}}
\newcommand{\oh}[1]{\widehat{ #1}}

\newcommand{\ab}[1]{\left( #1\right)}

\newcommand{\s}[1]{{#1}^{\mathrm{sym}}}
\newcommand{\p}{{q}}

\newcommand{\Hm}[1]{\leavevmode{\marginpar{\tiny%
$\hbox to 0mm{\hspace*{-0.5mm}$\leftarrow$\hss}%
\vcenter{\vrule depth 0.1mm height 0.1mm width \the\marginparwidth}%
\hbox to 0mm{\hss$\rightarrow$\hspace*{-0.5mm}}$\\\relax\raggedright
#1}}}

\begin{document}

\title[Volume Growth, Spectrum and Stochastic Completeness]{Volume Growth, Spectrum and Stochastic Completeness of Infinite Graphs}

\author{Matthias Keller}
\address{Mathematisches Institut \\
Friedrich Schiller Universit{\"a}t Jena \\
07743 Jena, Germany }
\email{m.keller@uni-jena.de}

\author{Daniel Lenz}
\address{Mathematisches Institut \\
Friedrich Schiller Universit{\"a}t Jena \\
07743 Jena, Germany }
\email{daniel.lenz@uni-jena.de}

\author{Rados{\l}aw K. Wojciechowski}
\address{York College of the City University of New York \\
Jamaica, NY 11451 \\ USA  }
\email{rwojciechowski@gc.cuny.edu}

\date{\today}

\begin{abstract}
We study the connections between volume growth, spectral properties and stochastic completeness of locally finite weighted graphs.  For a class of graphs with a very weak spherical symmetry we give a condition which implies both stochastic incompleteness and discreteness of the spectrum.  We then use these graphs to give some comparison results for both stochastic completeness and estimates on the bottom of the spectrum for general locally finite weighted graphs.
\end{abstract}

\maketitle

\section{Introduction}
The aim of this paper is to investigate the connections between volume growth, uniqueness of bounded solutions for the heat equation, and spectral properties for infinite weighted graphs.  To do so, we proceed in two steps.  We first establish these connections on a class of graphs with a very weak spherical symmetry.  We then give some comparison results for general graphs by using a notion of curvature.

A very general framework for studying operators on discrete measure spaces was recently established in \cite{KelLen09}.  We use this set up throughout with the additional assumption that the underlying weighted graphs are locally finite.

We first introduce the class of weakly spherically symmetric graphs which, compared with spherically symmetric graphs, can have very little symmetry.  Still, these graphs turn out to be accessible to a detailed analysis much as those with a full spherical symmetry.  More precisely, we can characterize them by the fact that:
\begin{itemize}
  \item Their heat kernels are spherically symmetric (Theorem~\ref{t:HK_wsp} in Section~\ref{s:HK}).
\end{itemize}
Moreover, for operators arising on these graphs we prove:
\begin{itemize}
  \item  An explicit estimate for the bottom of their spectrum and a criterion for the discreteness of the spectrum in terms of volume growth and the boundary of balls (Theorem~\ref{t:spectrum} in Section~\ref{s:spectrum}).
  \item  A characterization of stochastic completeness in terms of volume growth and the boundary of balls (Theorem~\ref{t:stochastic} in Section~\ref{s:stochastic}).
\end{itemize}
%\Hm{The Cheeger constant is the inf over these ratios isn't it?}
Both the estimate for the bottom of the spectrum and the condition for stochastic completeness involve the ratio of a generalized volume of a ball to its weighted boundary.
 %a quantity related to the isoperimetric or Cheeger's constant.
 In this sense, our estimates complement the classic lower bound on the bottom of the spectrum of the ordinary graph Laplacian in terms of Cheeger's constant given by Dodziuk in \cite{Do84} in the case of unbounded geometry.
%  Here, the notion of stochastic completeness is equivalent to the heat equation having a unique bounded solution \cite{KelLen09}.

These results give rise to examples of  graphs of polynomial volume growth which have positive bottom of the spectrum and are stochastically incomplete.  Therefore, as a surprising consequence, in the standard graph metric there are no direct analogues to the theorems of Grigor{\cprime}yan, relating volume growth and stochastic completeness \cite{Gri99}, and of Brooks, relating volume growth and the bottom of the essential spectrum \cite{Br81}, from the manifold setting.   For stochastic completeness this was already observed in \cite{Woj10}.  These examples are studied at the end of the paper, in Section~\ref{s:Application}.

We now turn to the second step of our investigation, i.e., the comparison of  general weighted graphs to weakly spherically symmetric ones. In this context, we provide:
\begin{itemize}
  \item heat kernel comparisons (Theorem~\ref{t:HK_comparison} in Section~\ref{s:HK})
  \item comparisons for the bottom of the spectrum (Theorem~\ref{t:spectral_comparison} in Section~\ref{s:spectrum})
  \item comparisons for stochastic completeness (Theorem~\ref{t:sc_comparison} in Section~\ref{s:stochastic}).
\end{itemize}
The heat kernel comparison, Theorem~\ref{t:HK_comparison}, is given in the spirit of results of Cheeger and Yau \cite{ChYau81}. However, in contrast to other works  done for graphs in this area, e.g. \cite{Ura97,Woj09}, as we use weakly spherically symmetric graphs, we require very little symmetry for our comparison spaces. The comparisons are then done with respect to a certain curvature (and, in the case of stochastic completeness, also to potential) growth.  We combine these inequalities with an analogue to a theorem of Li \cite{Li86}, which was recently proven in our setting in \cites{KLVW, HKLW}, to obtain comparisons for the bottom of the spectrum, Theorem~\ref{t:spectral_comparison}.  %\Hm{Why mention Ichiharas name - do you hope for him as a referee? Are you applying in Japan?}
The spectral comparisons give some analogues to results of Cheng \cite{Ch75} and extend inequalities found for graphs in \cites{Br91, Ura99, Zuk99}.  The comparison results for stochastic completeness, Theorem~\ref{t:sc_comparison}, are inspired by work of Ichihara \cite{Ich82} and are found in Section~\ref{s:stochastic}.

The article is organized as follows: in the next section we introduce the set up. This is followed by the heat kernel theorems, Theorem~\ref{t:HK_wsp} and Theorem~\ref{t:HK_comparison}, in Section~\ref{s:HK}, the spectral estimates, Theorem~\ref{t:spectrum} and Theorem~\ref{t:spectral_comparison}, in Section~\ref{s:spectrum} and the considerations about stochastic completeness, Theorem~\ref{t:stochastic} and Theorem~\ref{t:sc_comparison}, in Section~\ref{s:stochastic}.  The proofs are given within each section.  In Section~\ref{s:Application}, we discuss the applications to unweighted graph Laplacians and give the examples of polynomial volume growth announced above.  Finally, in Appendix \ref{s:Appendix}, we prove some general facts concerning commuting operators which are used in the proof of Theorem~\ref{t:HK_wsp}.

\section{The set up and basic facts} \label{s:Setup}
Our basic set up, which is included in \cites{KelLen09, KelLen10}, is as follows: let $V$ be a countably infinite set and $m:V\to(0,\infty)$ be a measure of full support.  Extending $m$ to all subsets of $V$ by countable additivity, $(V,m)$ is then a measure space. The map $b : V \times V \to [0, \infty)$, which characterizes the edges, is symmetric and has zero diagonal.  If $b(x,y)>0$, then we say that $x$ and $y$ are \emph{neighbors},  writing $x \sim y$, and think of $b(x,y)$ as the weight of the edge connecting $x$ and $y$. Moreover, let $c:V\to[0,\infty)$ be a map which we call the \emph{potential} or  \emph{killing term}. If $c(x)>0$, then we think of $x$ as being connected to an imaginary vertex at infinity with weight $c(x)$. We call the quadruple $(V,b,c,m)$  a \emph{weighted graph}. Whenever  $c\equiv 0$, we denote the weighted graph $(V,b,0,m)$ as the triple $(V,b,m)$. If, furthermore, $b:V\times V\to\{0,1\}$, then we speak of $(V,b,m)$ as an \emph{unweighted graph}.

We say that a weighted graph is \emph{connected} if, for any two vertices $x$ and $y$, there exists a sequence of vertices $(x_i)_{i=0}^n$ such that $x_0 = x$, $x_n = y$ and $x_i \sim x_{i+1}$ for $i=0,1, \ldots, n-1.$  We say that a weighted graph is \emph{locally finite} if every $x\in V$ has only finitely many neighbors, i.e.,  $b(x,y)$ vanishes for all but finitely many $y\in V$.

\bigskip

Throughout the paper we \textbf{assume} that all  weighted graphs $(V,b,c,m)$ in question are  connected and locally finite.

\bigskip

In this setting, weighted graph Laplacians and   Dirichlet forms on discrete measure spaces were recently studied in  \cite{KelLen09}, whose notation we closely follow (see also  the seminal work \cite{BD} on finite graphs and \cite{FOT} for background on general Dirichlet forms).
Let $C(V)$ denote the set of all functions from $V$ to $\R$ and let
\[ \ell^2(V,m) = \{ f\in C(V)  \ | \ \sum_{x \in V} f^2(x) m(x) < \infty \} \]
denote the Hilbert space of functions square summable with respect to $m$ with inner product given by
\[ \as{f,g} = \sum_{x \in V} f(x)g(x)m(x). \]
We then define the form $Q$ with domain of definition
\[ D(Q)=\ov{C_c(V)}^{ \as {\cdot,\cdot}_{Q}}, \]
where  $C_c(V)$ denotes the space of finitely supported functions, the closure is taken with respect to $\as{\cdot,\cdot}_{Q}:={\as{\cdot,\cdot} + Q(\cdot,\cdot)}$ in $\ell^{2}(V,m)$, and $Q$ acts by
\[ Q(f,g)=\frac{1}{2}\sum_{x,y\in V} b(x,y) \big( f(x) - f(y) \big) \big( g(x) - g(y) \big) + \sum_{x\in V} c(x) f(x)g(x).\]
Such forms are regular Dirichlet forms on the measure space $(V,m)$, see \cite{KelLen09}.

By general theory (see, for instance, \cite{FOT}) there is a selfadjoint positive operator $L$ with domain $ D(L) \subseteq \ell^2(V,m)$ such that $Q(f,g) = \as{Lf, g}$ for $f\in D(L)$ and $g\in D(Q)$.
By \cite[Theorem~9]{KelLen09} we know that  $L$ is a restriction of the \emph{formal Laplacian} $\Lf$ which acts as
\[ \Lf f(x) = \frac{1}{m(x)} \sum_{y \in V} b(x,y) \big( f(x)-f(y) \big) + \frac{c(x)}{m(x)} f(x) \]
and, as $(V,b,c,m)$ is locally finite, is defined for all functions in $C(V)$. Alternatively, as shown in \cite{KelLen09} as well, one can consider $L$ to be an extension - the so-called Friedrichs extension -  of the operator $L_0$ defined on $C_c (V)$ by
$$ L_0 g = {\ow L} g.$$
%Then, the operator  $\widetilde{L}$ on $C(V)$   is the adjoint of $L_0$ on $C_c (V)$ in the sense that
%$$ \sum_{x\in V}  \Lf f(x)  g (x) m(x) = \sum_{x\in V} f(x) L_0 g (x) m(x)$$
% holds for all $f\in C(V)$ and $g\in C_c (V)$ (as can be seen by a direct calculation).

%Let us also mention that, by  \cite{KelLen10}*{Theorem 11},  the operator $L$ is bounded if and only if there exists a constant $C$ such that for all vertices $x$
%\begin{equation*}\label{bounded}
% \frac{1}{m(x)} \big( b(x) + c(x) \big) \leq C.
%\end{equation*}

\bigskip

Two very prominent examples from the unweighted setting are the \emph{graph Laplacian} $\Lp$ given by the additional assumptions that $b:V\times V\to \{0,1\}$, $c \equiv 0$, and $m \equiv 1$ so that
\[ \Lp f(x) = \sum_{y \sim x} \big( f(x) - f(y) \big) \]
and the \emph{normalized graph Laplacian} $\Ln$ given by the additional assumptions that $b:V\times V\to \{0,1\}$, $c \equiv 0$, and $m \equiv \deg$  so that
\[ \Ln f(x) = \frac{1}{\deg(x)} \sum_{y \sim x} \big( f(x) - f(y) \big). \]
Here, $\deg(x) = | \{ y \ | \ y \sim x \} |$ for $x\in V$, where $|\{\cdot \}|$ denotes the cardinality of a set, is finite for all $x\in V$ by the local finiteness assumption.

In the following sections, we will compare our results to those for $\Lp$ and $\Ln$ found in the literature. Furthermore, we will illustrate some of our results for the operator $\Lp$ at the end of the paper in Section~\ref{s:Application}.  We note that $\Lp$ is bounded on $\ell^2(V) := \ell^2(V,1)$ if and only if $\deg$ is bounded, while $\Ln$ is always bounded on $\ell^2(V,\deg)$.

\bigskip

We will often fix a vertex $x_0$ and consider spheres and balls
\[S_r= S_r(x_0) = \{ x \ | \ d(x,x_0) = r \} \quad \tu{ and } \quad  B_r=B_r(x_0) = \bigcup_{i=0}^r S_i(x_0) \]
around $x_0$ of radius $r$.  Here, $d(x,y)$ is the usual combinatorial metric on graphs, that is, the number of edges in the shortest path connecting $x$ and $y$.

\begin{definition}
The \emph{outer} and \emph{inner} \emph{curvatures} $\ka_{\pm}:V\to[0,\infty)$ are given by
\begin{align*}
\ka_{\pm}(x)=\frac{1}{m(x)}\sum_{y\in S_{r\pm1}}b(x,y) \quad \textrm{ for } x \in S_r.
\end{align*}
\end{definition}

\begin{remark}
We refer to these quantities as curvatures as, for $c \equiv 0$, $\Lf d(x_0,\cdot)=\ka_{-}(\cdot) - \ka_{+}(\cdot)$ is often referred to as a curvature-type quantity for graphs, see \cite{DoKa88,Hu,Web10}.  Moreover, in light of our Theorem~\ref{t:HK_wsp} and \cite[Proposition~2.2]{ChYau81}, as well as our comparison results, Theorems~\ref{t:HK_comparison}, \ref{t:spectral_comparison}, and \ref{t:sc_comparison}, and their counterparts in the manifold setting which can be found in \cite{ChYau81, Ch75, Ich82}, it seems reasonable to relate $\ka_{\pm}$ to a type of curvature on a manifold.

Several other notions of curvature have been introduced for planar graphs \cite{Hi01, BaPe01, BaPe06}, cell complexes \cite{For}, and general metric measure spaces \cites{BStu, LV, Oll, StuI, StuII}.  See also the recent work on Ricci curvature of graphs \cite{BJL, JL, LY}.  It is not presently clear how these different notions of curvature are related.
\end{remark}

\begin{definition}
We call a function $f:V\to \R$  \emph{spherically symmetric} if  its values depend only on the distance to $x_0$, i.e., if $f(x) = g(r)$ for $x\in S_r (x_0)$ for some function $g$ defined on $\N_0 = \{0, 1, 2, \ldots \}$. In this case, we will often write $f(r)$ for $f(x)$ whenever $x\in S_r(x_0)$ and set, for convenience, $f(-1)=0$.
\end{definition}

Let $\A$ be the operator on $C(V)$  that averages a function over a sphere around $x_0$, i.e., \begin{align*}
(\A f)(x)=\frac{1}{m(S_{r})}\sum_{y\in S_{r}}f(y)m(y)
\end{align*}
for $f:V\to\R$ and $x\in S_{r}$.   This operator is a projection whose  range  is the spherically symmetric functions. In particular, a function $f$ is spherically symmetric if and only if $\A f=f$.
Moreover, the restriction $A$ of $\A$  to $\ell^{2}(V,m)$ is bounded and symmetric. Therefore, $A$ is an orthogonal projection.

\begin{definition}\label{sphericalgraphs}
Let the \emph{normalized potential} $\p:V\to[0,\infty)$ be given by $$\p=\frac{c}{m}.$$
We call a weighted graph $(V,b,c,m)$ \emph{weakly spherically symmetric} if it contains a vertex $x_0$ such that $\ka_{\pm}$ and $\p$ are spherically symmetric functions.
\end{definition}

We call the vertex $x_0$ in the definition above the \emph{root} of $(V,b,c,m)$.
\begin{remark}
We will often suppress the dependence on the vertex $x_0$. Mostly, we will denote weakly spherically symmetric graphs by $(\s V,\s b,\s c,\s m)$ although $\s b$, $\s c$ and $\s m$ might not have any obvious symmetries at all.  Furthermore, we will denote, as needed, the corresponding curvatures and potential by $\s \ka_\pm$ and $\s q$ and the Laplacian on such a graph by $\s L$.
\end{remark}

For a weakly spherically symmetric graph the operator $\Lf$ acts on a spherically symmetric function $f$ by
\begin{align*}
\Lf f(r)=\ka_{+}(r)(f(r)-f(r+1))+\ka_{-}(r)(f(r)-f(r-1))+\p(r)f(r).
\end{align*}
Moreover, a straightforward calculation yields that
\begin{equation} \label{eq:curvature and spheres}
\ka_{+}(r)m(S_{r}) = \ka_-(r+1) m(S_{r+1}) \quad\mbox{for all $r\in \N_0$}.
\end{equation}

%If furthermore the measure space is spherically symmetric, then,
%\begin{align*} \Lf f(r)=\frac{1}{m(r)}(k_{+}(r)(f(r)-f(r+1))+k_{-}(r)(f(r)-f(r-1)) +c(r)f(r)). \end{align*}

\smallskip

Let us give some examples to illustrate the definition of weakly spherically symmetric graphs. %The example second example is of particular importance as it characterizes weakly spherically symmetric graphs on spherically symmetric measure spaces.

\begin{example}
(a) We call a  weighted graph $(V,b,c,m)$ \emph{spherically symmetric} with respect to $x_0$  if for each $x,y\in S_r(x_0)$, $r\in\N_0$, there exists a weighted graph automorphism which leaves $x_0$ invariant and maps $x$ to $y$. In this case, the weighted graph is weakly spherically symmetric.

(b) If the functions $k_{\pm}:=\ka_{\pm}m$, the potential $c$ and the measure $m$ are all spherically symmetric functions, then the weighted graph is  weakly spherically symmetric.
On the other hand, given that the measure $m$ is spherically symmetric and the graph is weakly spherically symmetric, then $k_{\pm}$ and $c$ must be spherically symmetric.
\end{example}

\begin{center}
\begin{figure}[!h]
%\vspace{}
\includegraphics[scale=0.42]{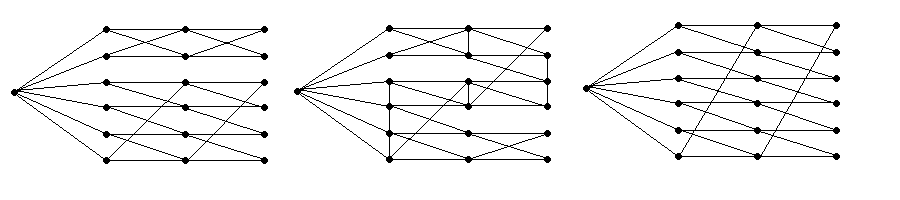}
\caption{The first two unweighted graphs are only weakly spherically symmetric while the third one is spherically symmetric.}
\label{f:wsp_vs_sp}
\end{figure}
\end{center}

\begin{remark}
The second example shows that there are very little assumptions on the symmetry of the geometry in the  weakly spherically symmetric case. The difference is illustrated in Figure~\ref{f:wsp_vs_sp}. There, unweighted graphs, that is, with $c \equiv 0$ and $b: V \times V \to \{0,1\}$, and constant measure are plotted up to the third sphere. The first and second graphs are only weakly spherically symmetric, while the third one is spherically symmetric. However, given the lack of assumptions on connections within a sphere and the structure of connections between the spheres, the freedom in the weakly spherically symmetric case is even much greater than illustrated in the figure. Indeed, we do not have any assumptions on the vertex degree as long as the outer and inner curvatures are constant on each sphere.
\end{remark}

\section{Heat Kernel Comparisons}\label{s:HK}

\subsection{Notations and definitions}

Let a weighted graph $(V,b,c,m)$ be given.  For the operator $L$ acting on $D(L)\subseteq\ell^2(V,m)$ we know, by the discreteness of the underlying space, that there exists a map
$$p:[0,\infty)\times V\times V\to\R,$$
which we call the \emph{heat kernel} associated to $L$, with
\[ e^{-tL}f(x)=\sum_{y\in V} p_t(x,y)f(y)m(y)\]
for all $f\in \ell^2(V,m)$.  Here, $e^{-tL}$ is the operator semigroup of $L$ which is defined via the spectral theorem. By direct computation, one sees that $p_{t}(x,y)=e^{-tL}\de_{y}(x)$, for $x,y\in V$ and $t\geq0$, where $\de_y$ is a $\ell^{1}$-normalized delta function, i.e., $\de_{y}(x)=\frac{1}{m(y)}$ if $x=y$ and zero otherwise.
\smallskip

\begin{definition}
We say that the heat kernel $p$ of an operator $L$ is \emph{spherically symmetric} if there is a vertex $x_0$ such that the averaging operator $A$ and the semigroup $e^{-tL}$ commute for all $t\geq0$.
\end{definition}
In particular, in this case, the function $p_{t}(x_0,\cdot)$ is spherically symmetric for each $t$ and, whenever this is the case, we write
\begin{align*}
    \s p_{t}(r)=p_{t}(x_0,x)
\end{align*}
for $x\in S_r$ and $r\in\N_0$.

\bigskip

In order to compare a general weighted graph with a weakly spherically symmetric one, we introduce the following terminology.

\begin{definition}\label{d:C_comp}
A weighted graph $(V, b,c,m)$ has \emph{stronger \tu{(respectively,} weaker\tu{)} curvature growth} with respect to $x_0 \in V$ than a weakly spherically symmetric graph $(\s V, \s b,\s c,\s m)$ with root $o$ if,  $m(x_0) = \s m(o)$ and if, for all $r \in \N_0$ and $x\in S_r\subset V$,
\begin{align*}
\ka_+(x)\geq \s \ka_+(r) \; &\tu{ and } \; \ka_-(x)\leq \s \ka_-(r)\\
(\tu{respectively, }\;\ka_+(x)\leq \s \ka_+(r) \; & \mbox{ and }\; \ka_-(x)\geq \s \ka_-(r)).
\end{align*}
\end{definition}

%With slight abuse of notation  we will denote the root vertex of $\s V$ and the respective vertex in  $V$ both by $x_0$. In fact, we will often suppress the vertex $x_0$ in our notation altogether.

\begin{remark}
The assumption $m(x_0) = \s m( o)$ is a normalization condition, necessary for our comparison theorems below, which states that the same amount of heat enters both graphs at the root.  This follows since, in general, $p_0(x,y) = \frac{1}{m(x)}$ if $x=y$ and 0 otherwise.
\end{remark}

\subsection{Theorems and remarks}

There are two main results about heat kernels which are proven in this section. The first one, which is an analogue for Proposition 2.2 in \cite{ChYau81} from the manifold setting, is that weakly spherically symmetric graphs can be characterized by the symmetry of the heat kernel.

\begin{theorem}\label{t:HK_wsp}\emph{(Spherical symmetry of heat kernels)} A weighted graph  $(V,b,c,m)$ is  weakly spherically symmetric  if and only if the heat kernel is spherically symmetric.
\end{theorem}

\begin{remark}
It is clear that the heat kernel on a spherically symmetric graph is spherically symmetric.  However, the theorem also implies that the heat kernels of all the graphs illustrated in Figure~\ref{f:wsp_vs_sp} are spherically symmetric.
\end{remark}

\smallskip
The second main result of this section is a  heat kernel comparison between weakly spherically symmetric graphs and general weighted graphs.  These comparisons were originally inspired by \cite{ChYau81} and can be found in \cite{Woj09} for the unweighted graph Laplacian and trees.  See also \cite{Ura97} for related results in the case of the unweighted graph Laplacian, regular trees, and heat kernels with a discrete time parameter.

\begin{theorem}\label{t:HK_comparison}\emph{(Heat kernel comparison with weakly spherically symmetric graphs)}
If a weighted graph $(V,b,m)$ with heat kernel $p$ has stronger (respectively, weaker) curvature growth than a weakly spherically symmetric graph  $(\s V, \s b,\s m )$ with heat kernel $\s p$, then for $x\in S_{r}(x_0)\subset V$, $r\in\N_0$, and $t\geq0$,
\begin{align*}
\s p_{t}(r)\geq p_{t}(x_0,x) \quad  \left(\mbox{respectively, }\s p_{t}(r)\leq p_{t}(x_0,x)\right).
\end{align*}
\end{theorem}

\subsection{Proofs of Theorems~\ref{t:HK_wsp} and \ref{t:HK_comparison}}
We start by developing the ideas necessary for the proof of Theorem~\ref{t:HK_wsp}.
We first characterize the class of weakly spherically symmetric graphs using $\A$, the averaging operator, in the following way.

\begin{lemma}\label{l:characterization} Let  $(V,b,c,m)$ be a  weighted graph and $x_0\in V$. Let $\A$ be the averaging operator associated to $x_0$.  Then, the following assertions  are equivalent:
\begin{itemize}
\item [(i)] $(V,b,c,m)$ is weakly spherically symmetric, i.e., $\ka_{\pm}$ and  $\p$ are spherically symmetric functions.
\item [(ii)] $\ow L$ commutes with $\A$, i.e., $\A \Lf f = {\ow L} \A f$ for all $f\in C (V)$.
\item [(iii)] $L_0$ commutes with $A$ on $C_c(V)$, i.e., $A L_0 g = L_0 A g $ for all $g \in C_c(V)$.
\end{itemize}
\begin{proof}
The direction (i)$\Longrightarrow$(ii) follows by straightforward computation using the formulas
$\ka_+(r)m(S_{r}) = \ka_-(r+1)m(S_{r+1})$, see (\ref{eq:curvature and spheres}).

As $A$ and $L$ are restrictions of $\A$ and $\Lf$, their matrix elements agree. This gives the direction (ii)$\Longrightarrow$(iii).

We finally turn to (iii)$\Longrightarrow$(i). The function   $1_{S_r}$, which is one on $S_r$ and zero elsewhere, satisfies
\begin{align*}
L_0 A1_{S_r}(x)=L_0 1_{S_r}(x)
&=\left\{\begin{array}{ll}
\ka_{+}(x)+\ka_{-}(x)+\p(x) & \mbox{if } x\in S_{r}\\
-\ka_{\mp}(x)& \mbox{if } x\in S_{r\pm 1}\\
0&\mbox{otherwise}\end{array}\right.
\end{align*}
and
\begin{align*}
A L_0 1_{S_r}(x)=&\left\{\begin{array}{ll}
\frac{1}{m(S_{r})}\sum_{y\in S_r}(\ka_{+}(y)+\ka_{-}(y)+\p(y))m(y) & \mbox{if }x\in S_{r}\\
-\frac{1}{m(S_{r\pm1})}\sum_{y\in S_{r\pm1}}\ka_{\mp}(y)m(y)& \mbox{if }x\in S_{r\pm 1}\\
0&\mbox{otherwise}.\end{array}\right.
\end{align*}
By (iii), these two expressions must be equal which yields that $\ka_{\pm}$ and $\p$ must be spherically symmetric functions which is (i).
\end{proof}
\end{lemma}

%The following lemma is standard but we give it for the sake of completeness.
%\begin{lemma}\label{l:commute} If a bounded selfadjoint operator  $B$ and an orthogonal projection $P$ on a Hilbert space $H$ commute on $\ran P$, then they commute on $H$.
%\begin{proof}Let $v, w\in H$. Then,
%\begin{align*}
%\as{BPv,w}=\as{BP^2v,w}=\as{PBPv,w}=\as{v,PBPw}=\as{v,BPw}=\as{PBv,w}.
%\end{align*}
%\end{proof}
%\end{lemma}

From Corollary \ref{main-appendix} in Appendix \ref{s:Appendix} with $\mathcal{H} = \ell^2 (V,m)$ and $D_0 = C_c (V)$ we can now  infer the following statement.  Recall that $A$ is the restriction of $\A$ to $\ell^2(V,m)$ which is an orthogonal projection onto the subspace of spherically symmetric functions.

\begin{lemma}\label{l:characterization2} Let  $(V,b,c,m)$ be a  weighted graph. Then, the following assertions are equivalent:
\begin{itemize}
\item [(i)] $ A L_0 f = L_0 A f$ for all $f\in C_c (V)$.
\item [(ii)] $A$ maps $D(L)$ into $D(L)$ and $A L f = L A f$ for all $f\in D(L)$.
    \item [(iii)]   $e^{-tL}$ commutes with $A$  for all $t\geq0$.
%\item [(iv)]   $(L+\al)^{-1}$ commutes with $A$  for all $\al>0$.
\end{itemize}
\end{lemma}
%\begin{proof}The implication (i)$\Rightarrow$(ii) is clear. From (ii) it follows that $L^{(D)}_i$ and commutes with the restriction $A_i$ of $A$ to $\ran A_i\subseteq \ell^{2}(B_i,m_i)$.  By Lemma~\ref{l:commute} the operators, then, commute on $\ell^{2}(B_i,m_i)$. This implies that, $e^{-tL^{(D)}_i}$ and $A_i$ commute on $\ell^{2}(B_i,m_i)$. By going over to the limit we get that $e^{-tL}$ and $A$ commute on $\ell^{2}(V,m)$ which is (iii). Now, the formula $(L+\al)^{-1}=\int_{0}^{\infty}e^{-t\al}e^{-tL}dt$ yields the implication (iii)$\Rightarrow$(iv). Finally, given (iv), we get for $f\in D(L)$
%\begin{align*}
%ALf=A\lim_{\al\to\infty}(\al-\al^{2}(L+\al)^{-1})f =\lim_{\al\to\infty}(\al-\al^2(L+\al)^{-1})Af =LAf.
%\end{align*}
%As the left hand side exist, so does the right hand side. This in particular implies that $Af$ is in the %domain $D(L)$ of the generator of $(L+\al)^{-1}$.
%\end{proof}
%\end{lemma}

\begin{remark} A  proof of the implication (i)$\Longrightarrow$(ii) in Lemma \ref{l:characterization2} could also be based on Proposition \ref{useful} from the appendix and the approximation of the heat kernel $p$ by the reduced kernels $p^{i}$ on  $\mathcal{H}_i := \ell^2 (B_i, m_i)$ discussed below.
\end{remark}

\begin{proof}[Proof of Theorem~\ref{t:HK_wsp}]
To prove Theorem~\ref{t:HK_wsp} simply combine Lemmas~\ref{l:characterization} and \ref{l:characterization2}.
\end{proof}

\smallskip

The following construction of the heat kernel, first presented in the continuous setting in \cite{Do83} and carried over to our set-up in \cite{KelLen09}, will be crucial.  Let $x_0\in V$ and $B_i=B_i(x_0)$ denote the corresponding distance balls. By the connectedness of the graph, $V=\bigcup_{i=0}^{\infty}B_i$ and, as $B_i\subseteq B_{i+1}$, the balls are an increasing exhaustion sequence of the graph $(V,b,c,m)$.  For $i\in\N_0$, let $m_i$ be the restriction of $m$ to $B_i$ and consider the restriction  $L_i^{(D)}$ of the Laplacian $L$ to the finite dimensional space $\ell^2(B_i,m_{i})$ with Dirichlet boundary conditions. This operator can be defined by restricting the form $Q$ to $C_c({B_{i}})$ and taking the closure in $\ell^{2}({B_i},m_{i})$ with respect to $\as{\cdot,\cdot}_{Q}$. It turns out, (see \cite{KelLen09}), that
\[ L_i^{(D)} f (x) =\Lf f (x) \quad \mbox{for $f\in \ell^{2}({B_{i}},m_{i})$ and $x\in B_i$.}\]
This just means that $L_i^{(D)} = \pi_i \Lf \iota_i$
for the canonical injection $\iota_i: \ell^2 (B_i, m_i)\longrightarrow \ell^2 (V,m)$ acting as extension by zero and $\pi_i$, the adjoint of $\iota_i$.

We let $p^{i}$ denote the heat kernel of $L_i^{(D)}$ which is extended to $V\times V$ by zero.  That is,
\[ p_{t}^{i}(x,y)= \left\{
\begin{array}{ll}
e^{-tL_{i}^{(D)}}\de_y(x) & \mbox{if } x,y\in B_{i}, t\geq0 \\
0 & \mbox{otherwise.}
\end{array} \right. \]
Here, $\de_{y}(x)$ is the $\ell^1$-normalized delta function as before. We call $p^{i}$ the \emph{restricted heat kernels} and note that each $p^i$ satisfies
$$\big( \Lf + \tpdt \big) p_t^i(x,y)= 0\qquad\mbox{for all $x,y\in {B_{i}}, t \geq 0$}.$$
Furthermore, by \cite{KelLen09}*{Proposition~2.6},
$$\lim_{i \to \infty} p_t^i(x,y) = p_t(x,y)\qquad\mbox{for all $x,y\in V$}.$$
Therefore, to prove a property of the heat kernel it often suffices to prove the corresponding property for the reduced heat kernels and then pass to the limit.  This is used repeatedly below.

\medskip

In order to prove Theorem~\ref{t:HK_comparison}, we need a version of the minimum principle for the heat equation in our setting.  For $U \subseteq V$ we let $U^c = V \setminus U$.

\begin{lemma} \label{l:maximum}\emph{(Minimum principle for the heat equation)} Let $(V,b,c,m)$ be a weighted graph, $U \subset V$  a connected proper subset and $u: V \times [0, T] \to \R$ such that $u(x,\cdot)$ is continuously differentiable for every $x\in U$.
Suppose that
\begin{itemize}
  \item [(a)] $(\Lf + \tpdt ) u \geq0$ on $U \times [0,T]$,
  \item [(b)] the negative part $u_{-}:=\min\{u, 0\}$ of $u$ attains its minimum on $U\times[0,T]$,
  \item  [(c)] $u\geq0$ on $(U^{c}\times[0,T]) \cup (U \times \{0\})$.
\end{itemize}
Then, $u\geq0$ on $U \times [0,T]$.
\end{lemma}
\begin{proof}
Let $(x,t)$ be the minimum of $u_-$ on $U \times [0,T]$.  If $u(x,t) \geq 0$, then we are done.  Furthermore, by (c), we may assume that $t>0$. Therefore, suppose that $u(x,t)<0$ where $t>0$. As $(x, t)$ is a minimum for $u$, we have that $(\tpdt u)(x,t)=0$ if $t\in(0,T)$ and $(\tpdt u)(x,t)\leq 0$ if $t=T$.
Since $(x,t)$ is also a minimum with respect to $x$ it follows  that
\begin{align*}
\big( \Lf + \tpdt \big) u(x,t) &  \leq \Lf u (x,t) =  \frac{1}{m(x)} \sum_{y\in V} b(x, y) \big( u(x,t) - u(y,t) \big)+\frac{c(x)}{m(x)}u(x,t)  \leq 0.
\end{align*}
Therefore, by (a), $\Lf u(x,t) = 0$ so that $u(x,t) = u(y,t)<0$ for all $y \sim x$.  Repeating the argument we eventually reach some $y \not \in U$ contradicting (c).
\end{proof}

We also need the following extension of Lemma~3.10 from \cite{Woj09} which states that the heat kernel on a weakly spherically symmetric graph decays with respect to $r$.  The proof can be carried over directly to our situation so we omit it here.

\begin{lemma}\label{l:decay}\emph{(Heat kernel decay, \cite[Lemma~3.10]{Woj09})}
Let $( V, b, m)$ be a weighted graph, $p$ be the heat kernel and $p^{i}$ be the restricted kernels of $B_{i}$. Assume that $p_t^{i}(x_0,\cdot)$ is a spherically symmetric function for all $t\geq0$ with respect to some vertex $x_0$. Then, given $0\leq r \leq i$, we have for all $t > 0$
\[  p_t^{i}(r) > p_t^{i}(r+1) \]
and, in general, for all  $r \in \N_0$ and $t \geq 0$
\begin{equation*} \label{decrease}
p_t(r) \geq  p_t(r+1).
\end{equation*}
\end{lemma}

\smallskip

With these preparations, we can now prove the second main theorem as follows:

\begin{proof}[Proof of Theorem~\ref{t:HK_comparison}]
Let $p^{\mathrm{sym,i}}$ be the restricted heat kernels  of the weakly spherically symmetric graph $(\s V,\s b,\s m)$ given in the statement of the theorem. On the general weighted graph $(V,b,m)$, we define the functions $\rho^{i}_t:V\to\R$, $t\geq0$, by
\begin{align*}
    \rho^{i}_{t}(x):= p_{t}^{\mathrm{sym,i}}(r)\quad\mbox{ for $x\in S_r(x_0)$, $0\leq r\leq i$}
\end{align*}
and $\rho_t^{i}(x):=0$ otherwise.

Under the assumptions of stronger curvature growth and using the heat kernel decay, Lemma~\ref{l:decay}, it follows from the action of $\Lf$ on spherically symmetric functions given after Definition \ref{sphericalgraphs}  that, for all $0\leq r\leq i$ and $x \in S_r(x_0)\subset V$,
\begin{align*}
\Lf \rho_t^{i}(x)=& \ka_{+}(x) ( p_{t}^{\mathrm{sym,i}}(r) - p_{t}^{\mathrm{sym,i}}(r+1) ) + \ka_{-}(x) (p_{t}^{\mathrm{sym,i}}(r)-p_{t}^{\mathrm{sym,i}}(r-1))\\
\geq& {\s \ka_+(r)} (p_{t}^{\mathrm{sym,i}}(r) - p_{t}^{\mathrm{sym,i}}(r+1)) +{\s \ka_-(r)}(p_{t}^{\mathrm{sym,i}}(r) - p_{t}^{\mathrm{sym,i}}(r-1))\\
=&\s\Lf p_{t}^{\mathrm{sym,i}}(r).
\end{align*}
Hence,
$(\Lf+\tpdt)\rho_t^{i}(x)\geq(\s\Lf+\tfrac{d}{dt})p_{t}^{\mathrm{sym,i}}(r)=0$.

Let $p^{i}$ be the restricted  heat kernels of the general weighted graph $(V,b,m)$ so that  $(\Lf+\tpdt)p_t^{i}(x_0,\cdot)=0$ on $B_i\subset V$ and let $u(\cdot,t) = \rho_t^i(\cdot) - p_t^i(x_0,\cdot)$.  It follows that
\[ \big( \Lf + \tpdt \big) u(x,t) \geq 0  \]
on $B_i \times [0,T]$.   By compactness, the negative part of $u(x,t)$ attains its minimum on $B_i\times[0,T]$.  Furthermore, as $\rho_{t}^{i}=0$ on $B_i^c$ by definition, we have $u(x,t) = 0$ on $B_i^c$ and $\rho_0^i(\cdot) = p_0^i(x_0,\cdot)$ which is $\frac{1}{m(x_0)} = \frac{1}{\s m(o)}$ for $x_0$ and 0 otherwise.  Hence, by the minimum principle,  Lemma \ref{l:maximum}, we get that
\[ p_{t}^{\mathrm{sym,i}}(r)=  \rho_t^i(x) \geq p_t^i(x_0,x) \quad\mbox{for } x\in B_i. \]
The desired  result now  follows by letting $i \to \infty$.  The inequality in the case of weaker curvature growth is proven analogously.
\end{proof}

\section{Spectral estimates}\label{s:spectrum}
In this section we give an estimate for the bottom of the spectrum and a criterion for discreteness of the spectrum in the weakly spherically symmetric case. We then use the heat kernel comparisons obtained above to give estimates on the bottom of the spectrum for general weighted graphs.

\subsection{Notations and definitions}
Let a weighted graph $(V,b,m)$ be given. Let $\si(L)$ denote the spectrum of $L$ and
\[ \lm_0:=  \lm_0(L):=\inf\si(L). \]
We call $\lm_0$ the \emph{bottom of the spectrum} or the \emph{ground state energy}.  The ground state energy can be obtained  by the Rayleigh-Ritz quotient (see, for instance, \cite{RS78}) as follows:
\begin{equation*}\label{bottom}
\lm_0= \inf_{f \in D(Q)} \frac{Q(f,f)}{\as{f,f}} = \inf_{ f \in C_c(V) } \frac{\as{Lf,f}}{\as{f,f}}
\end{equation*}
where the last equality follows since $C_c(V)$ is dense in $D(Q)$ with respect to $\as{\cdot,\cdot}_{Q}={\as{\cdot,\cdot} + Q(\cdot,\cdot)}$ as discussed  in Section \ref{s:Setup}.  Furthermore, the spectrum of an operator may be decomposed as a disjoint union as follows:
\[ \si(L) = \sd(L) \ \dot{\cup} \ \se(L) \]
where $\sd(L)$ denotes the \emph{discrete spectrum} of $L$, defined as the set of isolated eigenvalues of finite multiplicity, and $\se(L)$ denotes the \emph{essential spectrum}, given by  $\se(L)=\si(L)\setminus\sd(L)$.  We will use the notation $\lmess(L)$ to denote the bottom of the essential spectrum of $L$.

Fix a vertex $x_0$ and let $S_r = S_r(x_0)$ and $B_r=B_r(x_0)$.
\begin{definition}
 For $f:V\to\R$ and $r\in\N_0$, we define the \emph{weighted volume of a ball} by
\begin{align*}
V_f(r)=\sum_{x\in B_r}f(x)m(x).
\end{align*}
In particular, $V_1(r)=m(B_r)$. Moreover, for $r\in\N_0$, we let
\begin{align*}
\partial B(r) = \sum_{x \in S_r} \ka_+(x)m(x)
\end{align*}
be the \emph{measure of the boundary of a ball} which is the weight of the edges leaving the ball.  Note that, in the weakly spherically symmetric case, $\partial B(r) = \s \ka_+(r) \s m(S_r).$
\end{definition}

\subsection{Theorems and remarks}
The first main result of this section is an estimate on the bottom of the spectrum and a criterion for discreteness of the spectrum in terms of volume and boundary growth for weakly spherically symmetric graphs.

\begin{theorem}\label{t:spectrum}\emph{(Volume and spectrum)} Let $(\s V,\s b,\s m)$ be a weakly spherically symmetric graph. If
\[ \sum_{r=0}^\infty \frac{V_1(r)}{\partial B(r)} = a<\infty, \]
then
\[ \lm_0(\s L) \geq \frac{1}{a}  \ \tu{ and } \ \si(\s L) = \sd(\s L). \]
\end{theorem}
%\begin{remark} As  the potential does not enter the value of $a$ in Theorem \ref{t:spectrum}, the statement, more or less, considers the case $c\equiv0$. We note that it is trivial from (\ref{bottom}) that if, additionally, $\frac{\cs}{\ms} \geq C$, then $\lm_0(\s L) \geq \frac{1}{a} + C$. We discuss the difficulties of incorporating  the potential into $a$ after the proof which is given at the end of the section. \end{remark}

The second main result of this section is a comparison of the bottom of the spectrum of a general weighted graph and a weakly spherically symmetric one.

\begin{theorem}\label{t:spectral_comparison}\emph{(Spectral comparison)}
If a weighted graph $(V, b,m)$ has stronger (respectively, weaker) curvature growth than a weakly spherically symmetric graph $(\s V,\s b,\s m)$, then
\[ \lm_0(L) \geq\lm_0(\s L)\quad(\mbox{respectively, } \lm_0(L) \leq\lm_0(\s L)). \]
If $(\s V,\s b, \s m)$  satisfies
$\sum_{r=0}^\infty \frac{ V_1(r)}{\partial  B(r)} = a<\infty $
and $(V, b,m)$ has stronger curvature growth, then
\[ \lm_0(L) \geq \frac{1}{a} \ \tu{ and } \ \si(L) = \sd(L). \]
\end{theorem}

\smallskip

\begin{remarks}
Let us discuss these results in light of the present literature:

(a) In Section~\ref{s:Application}, we give examples of unweighted graphs with $m \equiv 1$ of polynomial volume growth satisfying the summability criterion above.  For such graphs, it follows that $\Lp$ has positive bottom of the spectrum as well as discrete spectrum.  This stands in clear contrast to the celebrated theorem of Brooks for Riemannian manifolds \cite{Br81} and results for $\Ln$ \cites{Fuj96a, Hi03}  since, for these, subexponential volume growth always implies that the bottom of the essential spectrum is zero.

(b) For statements analogous to Theorem~\ref{t:spectrum} for the Laplacian on spherically symmetric Riemannian manifolds see \cites{BP06, Ha09}.

(c) There are many examples of estimates for the bottom of the spectrum for the graph Laplacian $\Lp$ and normalized graph Laplacian $\Ln$, see, for example, \cites{Do84, DKe86, DoKa88, Mo88, BMS88,  Ura99, Ura00,KP, D06, DM}.  In particular, from the analogue of the Cheeger inequality found in \cite{Do84} it follows that $\lm_0(\Lp) = 0$ if $\frac{V_1(r)}{\partial B(r)} \to 0$ as $r \to \infty$.  Our Theorem~\ref{t:spectrum} complements these results by giving a lower bound for $\lm_0(\Lp)$ in the case of unbounded vertex degree. For  $\Ln$, our result is not applicable as $m(x) = \deg(x)$ implies that $V_1(r) \geq \partial  B(r)$.

(d) Discreteness of the spectrum of the graph Laplacians was studied in \cites{Fuj96b, Ura99, Kel10, Woj09}.  In our context, a characterization for  weighted graphs with positive Cheeger's constant at infinity to have discrete spectrum was recently given in \cite{KelLen10}. In Section~\ref{s:Application}, we discuss how our results complement these and are not implied by any of them, see Corollaries~\ref{c:polynom_spectrum} and \ref{c:edges}.

(e) With regards to Theorem \ref{t:spectral_comparison}, it was shown in \cite{Br91} that the bottom of the spectrum of the graph Laplacian $\Lp$ on a $k$-regular graph is smaller than that on a $k$-regular tree. This was generalized from $k$-regular to arbitrary graphs with degree bounded by $k$ in \cite{Ura99}, which also contains a corresponding lower bound.  Analogous statements for $\Ln$ were proven by different means in \cite{Zuk99}.
\end{remarks}

\subsection{Proofs of Theorems \ref{t:spectrum} and \ref{t:spectral_comparison}}
The proofs of the first statement in Theorem \ref{t:spectrum} and the second statement in Theorem \ref{t:spectral_comparison} are based on the following characterization for the bottom of the spectrum, which is sometimes referred to as a Allegretto-Piepenbrink type of theorem. We refer to \cite{HK}*{Theorem 3.1} for a proof and further discussion of earlier results of this kind.

\begin{proposition}\emph{(Characterization of the bottom of the spectrum, \cite[Theorem~3.1]{HK})} \label{positive}
Let $(V,b,c,m)$ be a weighted graph.  For $\al \in \R$ the following statements are equivalent:
\begin{itemize}
\item[(i)] There exists a non-trivial $v: V \to [0, \infty)$ such that $(\Lf + \al) v \geq 0$.
\item[(ii)] There exists $v: V \to (0, \infty)$ such that $(\Lf + \al) v = 0$.
\item[(iii)] $-\al \leq \lm_0(L).$
\end{itemize}
\end{proposition}

Therefore, to prove a lower bound on the bottom of the spectrum, it is sufficient to demonstrate a  positive (super-)solution to the difference equation above.  In the weakly spherically symmetric case, we will look for spherically symmetric solutions. %With slight abuse of notation we will write $v(r)$ for $v(x)$ when $x\in S_r$ for $r\in\N_0$ and vice versa.

We now state and prove the following lemma which generalizes a result of \cite{Woj10} and gives the existence of solutions for any initial condition. Note that we allow for a non-negative potential as we will also use this statement in the next section on stochastic completeness.

\begin{lemma}\emph{(Recursion formula for solutions)} \label{radialsolutions}
Let   $(\s V,\s b,\s c,\s m)$ be a weakly spherically symmetric graph and $\al\in\R$.  A spherically symmetric function $v$ is a solution to  $(\s {\ow L}+ \al ) v(r) = 0$ if and only if
\begin{equation*}\label{recursion}
v(r+1) - v(r) = \frac{1}{\partial  B(r)} \sum_{j=0}^r C_{ \p + \al}(j) v(j)
\end{equation*}
where $C_{\p  + \al}(j) = (\s \p(j) + \al)\s m (S_j)$.  In particular, $v$ is uniquely determined by the choice of $v(0)$. Consequently, if $v(0) >0$ and $\al>0$, then $v$ is a strictly positive, monotonously increasing solution.
\end{lemma}
\begin{proof}
%As $(\s V,\s m)$ is a spherically symmetric measure space, the functions $\s k_{\pm}$, $\s c$ and $\s m$ are spherically symmetric by the weak spherical symmetry of $(b,c)$.
The proof is by induction.  For $r=0$, $(\s {\ow L} + \al) v(0) =0$ gives
\[ (\s {\Lf} + \al ) v(0) = {\s \ka_+(0)} \big( v(0) - v(1) \big) + \left( \s\p(0) + \al \right) v(0) =0 \]
which yields the assertion.

Assume now that the recursion formula holds for $r-1$ where $r\geq1$.  Then,  $(\s {\ow L}+ \al)v(r)=0$ reads as
\begin{align*}
\s \ka_+(r) \big( v(r) - v(r+1) \big) + \s \ka_-(r) \big( v(r) - v(r-1) \big)   + \left( {\s \p(r)} + \al \right) v(r) = 0.
\end{align*}
Therefore,
\begin{align*}
v(r+1) - v(r) &= \frac{\s\ka_-(r)}{\s\ka_+(r)} \big( v(r) - v(r-1) \big) + \frac{1}{\s \ka_+(r)} \big(\s \p(r) + \al\big) v(r) \\
 &= \frac{\s \ka_-(r)}{\s \ka_+(r)} \left( \frac{1}{\s \ka_+(r-1) \s m(S_{r-1})} \sum_{j=0}^{r-1} C_{q + \al }(j) v(j) \right) \\
 & \qquad + \frac{1}{\s \ka_+(r)} \big( \s \p(r) + \al \big) v(r) \\
& = \frac{1}{\partial  B(r)} \sum_{j=0}^r  C_{q + \al}(j) v(j)
\end{align*}
by $\s \ka_+(r-1) \s m(S_{r-1}) = \s \ka_-(r) \s m(S_{r})$ as noted in (\ref{eq:curvature and spheres}).

Whenever $\al>0$, the right hand side of the recursion formula is positive from the assumption that $v(0) >0$ which gives the monotonicity statement.
\end{proof}

In order to prove the statements of Theorems~\ref{t:spectrum} and \ref{t:spectral_comparison} concerning the essential spectrum, we need to restrict our operator $L$ to the complements of balls. For $i\in  \N_0$, let $B_i^c = V \setminus B_i$ and $m_{i}^c$ be the restriction of $m$ to $B_i^{c}$.
We restrict the form $Q$ to $C_c(B_i^c)$ and take the closure in $\ell^{2}(B_i^c,m_i^c)$ with respect to $\as{\cdot,\cdot}_Q$. By standard theory, we obtain an operator on $\ell^{2}(B_i^c,m_i^c)$ which we call the restriction of $ L$ with Dirichlet boundary conditions and which we denote by $L_{i}^{(D)}$. Note that, in contrast to the previous section, these operators are now defined on the complement of balls and hence on infinite dimensional spaces.

\begin{lemma}\emph{(Existence of strictly positive solutions)} \label{l:existence_positive_solutions}
Let $(\s V,\s b,\s m)$  be  a weakly spherically symmetric graph. Suppose that $a=\sum_{r=0}^\infty \frac{ V_1(r)}{\partial B(r)}  <\infty$. Then, there exists a strictly positive, strictly monotone decreasing spherically symmetric solution $v$ on $\s V$ to $(\widetilde{L}^{\tu{sym}} -\frac{1}{a})v=0$ which satisfies
\begin{align*}
v(r+1) \geq  1 - \frac{1}{a} \sum_{j=0}^r \frac{V_1(j)}{\partial  B(j)}  \quad\mbox{ for all $r\in\N_0$.}
\end{align*}
Moreover, for all $i\in\N_0$, there exists a strictly positive, strictly monotone decreasing function $v_{i}$ on $B_i^{c}$ solving $( \widetilde{L}^{(D)}_{i} -\frac{1}{a_i})v_{i}=\s\ka_{-}(i+1)1_{S_{i+1}}$ which satisfies
$$v_{i}(r+1) \geq  1 - \frac{1}{a_{i}} \sum_{j=i+1}^r \frac{V^{i}_1(j)}{\partial B(j)}  \mbox{ for all $r\geq i+1$},$$
where $1_{S_{i+1}}(x)$ is $1$ for $x\in S_{i+1}$ and $0$ otherwise, $a_i = \sum_{j=i+1}^\infty \frac{V_1(j)}{\partial B(j) }$ and $V_1^i(j) = \s m(B_j\setminus B_{i})$.
\end{lemma}

\begin{proof}
For $\al = -\frac{1}{a}$ and $\s c \equiv 0$ the recursion formula of Lemma~\ref{radialsolutions}  reads as
\begin{equation*}\label{recursion2}
v(r+1) - v(r) = -\frac{1}{a \partial  B(r) } \sum_{j=0}^r C_{1}(j) v(j),
\end{equation*}
where $C_{1}(j)= \s m(S_{j}) $.
Hence, there exists a solution $v$ of the equation for $v(0)>0$.
In order to prove our assertion, we show by induction that for all $r \in \N_0$,
\begin{itemize}
\item[(i)] $v(r+1) < v(r)$
\item[(ii)] $v(r+1) \geq \left( 1 - \frac{1}{a} \sum_{j=0}^r \frac{V_1(j)}{\partial B(j)} \right) v(0)$
\item[(iii)] $v(r+1) > 0.$
\end{itemize}
For $r=0$, we get from the recursion formula above that
\[ v(1)- v(0) = - \frac{1}{a \partial B(0)} \s m(0) v(0) < 0\]
which gives (i).  Furthermore,
\[ v(1) = \left( 1 - \frac{V_1(0)}{a \partial B(0)} \right) v(0) \]
which gives (ii) and (iii) follows by the choice of $a$.

Now, suppose that (i), (ii), and (iii) hold for $r>0$.  Then, since $v(j) > 0$ for $j=0, 1, \ldots, r$, the recursion formula above yields that $v(r+1) - v(r) < 0$  which gives (i).  Furthermore,
\begin{align*}
v(r+1) &= v(r)  - \frac{1}{a \partial  B(r)} \sum_{j=0}^r C_1(j) v(j) > v(r) - \frac{V_1(r)}{a \partial  B(r)} v(0) \\
& \geq \left( 1 - \frac{1}{a} \sum_{j=0}^{r-1} \frac{V_1(j)}{\partial  B(j) } \right) v(0)   - \frac{V_1(r)}{a \partial  B(r)} v(0)   = \left( 1 - \frac{1}{a} \sum_{j=0}^r \frac{V_1(j)}{\partial  B(j)} \right) v(0)
\end{align*}
which yields (ii) and (iii) follows by the choice of $a$.

For the second statement, we  define the function $v_i$ on $B_i^c$ by $v_i(i+1)=1$ and
\[ v_i(r) =  v_i(r-1) - \frac{1}{a_i \partial  B(r-1)} \sum_{j=i+1}^{r-1} C_1(j) v_i(j) \quad \tu{ for } r > i+1. \]
By a direct calculation one checks that $( L_i^{(D)} - \frac{1}{a_i} ) v_i(i+1) = \ka_-(i+1)\geq0$ and, as in the proof of  Lemma~\ref{radialsolutions}, that $( L_i^{(D)} - \frac{1}{a_i} ) v_i(r) = 0$ for $r > i+1$.
Now, by the same arguments as above, one shows that $v_i$ is strictly monotone decreasing,  satisfies $v_i(r) \geq  1 - \frac{1}{a_i} \sum_{j = i+1}^{r-1} \frac{V_1^i(j)}{\partial  B(j)} $ and is strictly positive.
\end{proof}

\smallskip
%With these preparations, we can prove Theorem~\ref{t:spectrum}.

\begin{proof}[Proof of Theorem~\ref{t:spectrum}]
By Lemma~\ref{l:existence_positive_solutions}, there is a strictly positive solution to $(\widetilde{L}^{\tu{sym}} -\frac{1}{a})v=0$ where $a=\sum_{r=0}^\infty \frac{ V_1(r)}{\partial  B(r)}$.
This proves that $\lm_0(\s L) \geq \frac{1}{a}$ by the characterization of the bottom of the spectrum, Proposition~\ref{positive}.

We now show that $\si(\s L) = \sd(\s L)$. By standard theory (see, for example, Proposition~18 in \cite{KelLen10}) if follows that
\[ \lmess(\s L) = \lim_{i \to \infty} \lm_0(L_i^{(D)}). \]
By  Proposition~\ref{positive} and the second part of Lemma~\ref{l:existence_positive_solutions}, we have that $\lm_0(L_i^{(D)}) \geq \frac{1}{a_i}.$
%where $a_i = \sum_{r=i+1}^\infty \frac{V_1(r)}{\partial  B(r) }$.
Since $a_i \to 0$ as $i \to \infty$, it follows that $\lm_0(L_i^{(D)}) \to \infty$ as $i \to \infty$ so that $\se(\s L) = \emptyset$, that is, $\si(\s L) = \sd(\s L)$.
\end{proof}

\smallskip

In order to prove the spectral comparison, Theorem~\ref{t:spectral_comparison},  we need an analogue of the well-known theorem of Li which links large time heat kernel behavior and the ground state energy \cite{Li86, ChKa91, P04}.  It was recently proven for our setting in \cite{KLVW, HKLW}.

\begin{proposition}\label{t:Li}\emph{(Heat kernel convergence to ground state energy, \cite{KLVW, HKLW})} Let $(V,b,c,m)$ be a weighted graph. For all vertices $x$ and $y$,
\begin{equation*}
\lim_{t \to \infty} \frac{\ln p_t(x,y)}{t} = - \lm_0(L).
\end{equation*}
\end{proposition}

\smallskip
Theorem~\ref{t:spectral_comparison} now follows directly:

\begin{proof}[Proof of Theorem \ref{t:spectral_comparison}] The first statement follows by combining Theorem~\ref{t:HK_comparison} and Proposition~\ref{t:Li}. For the second statement, we first derive $\lm_0(L)\geq \frac{1}{a}$ from the first statement and Theorem~\ref{t:spectrum}. Now, let $v_i$ be the strictly positive monotone decreasing functions which solve $(\Lf^{(D)}-\frac{1}{a_i})v_i=\s\ka_-(i+1)1_{S_{i+1}}$ on $B_i^{c}\subseteq\s V$ with $a_i = \sum_{j=i+1}^\infty \frac{V_1(j)}{\partial B(j)}$ which exist according to Lemma~\ref{l:existence_positive_solutions} for all $i\geq0$. We choose $v_i$ to be normalized by letting $v_i(i+1)=1$. We define $w_i$ on $B_i^{c}\subseteq V$ via $w_i(x)=v_i(r)$ for  $x\in S_r$, $r\geq i+1$. Clearly, by the stronger curvature growth, $(\Lf_i^{(D)}-\frac{1}{a_i})w_i(x) \ge(\Lf_i^{\mathrm{sym},(D)}-\frac{1}{a_i})v_i(r)$ for $x\in S_{r}$, $r> i+1$. On the other hand, for $r=i+1$ and $x\in S_r$ we have, again by the stronger curvature growth, that
$$(\Lf_i^{(D)}-\tfrac{1}{a_i})w_i(x)\ge (\Lf_i^{\mathrm{sym},(D)}-\tfrac{1}{a_i})v_i(i+1) +(\ka_{-}(x)-\s \ka_{-}(i+1)) =\ka_{-}(x)\geq0$$ since $v_i(i+1)=1$.
Hence, $(\Lf_i^{(D)}-\frac{1}{a_i})w_i\ge0$ with $w_i$ strictly positive and we conclude that $\lm_0(L_{i}^{(D)})\geq \frac{1}{a_i}$ by Proposition~\ref{positive}. As $a_i\to 0$, we get, as in the proof of Theorem~\ref{t:spectrum}, that $\se(L)=\emptyset$ and, therefore, $\si(L)=\sd(L)$.
\end{proof}

\section{Stochastic completeness}\label{s:stochastic}

\subsection{Notations and definitions}
The study of the uniqueness of bounded solutions for the heat equation has a long history in both the discrete, see, for example, \cites{Fel57, Reu57}, and the continuous, see \cite{Gri99} and references therein, settings.  In recent years, there has been interest in finding geometric conditions for infinite graphs implying this uniqueness, see, for example, \cites{DM, D06, GHM, Hu, KelLen09, KelLen10, Web10, Woj08, Woj09, Woj10}.  In the  general setting of \cite{KelLen09} it is shown that this uniqueness is equivalent to several other properties as we discuss below.

Let a weighted graph $(V,b,c,m)$ be given.  We let $u_0: V \to \R$ be bounded and call $u : V \times [0, \infty) \to \R$ a \emph{solution of the heat equation with initial condition $u_0$} if, for all $x\in V$, $u(x, \cdot)$ is continuous on $[0, \infty)$, differentiable on $(0,\infty)$ and satisfies
\[ \left\{
\begin{array}{ll} \big( \Lf+ \tpdt \big)u(x,t) = 0 & \tu{ for } x \in V, \ t > 0,\\
  u(x, 0) = u_0(x) & \tu{ for } x \in V.
 \end{array} \right.  \]
The question of the uniqueness of bounded solutions for the heat equation on $(V,b,c,m)$ is then reduced to having $u \equiv 0$ be the only bounded solution for the heat equation with initial condition $u_0 \equiv 0$.

In order to study this question, the following function which was introduced in \cite{KelLen09} turns out to be essential. Let
\[ M_t(x) = e^{-tL}1(x) + \int_0^t e^{-sL} \p (x) ds, \]
where $\p = \frac{c}{m}$ is the normalized potential, $1$ denotes the function whose value is $1$ on all vertices and $e^{-tL}$ is the operator semigroup extended to the space of bounded functions on $V$. The function $e^{-sL}\p$ is defined as the pointwise limit along the net of functions $g\in C_c(V)$ such that $0\leq g\leq \p$, where the net is considered with respect to the natural ordering $g\prec h$ whenever $g\leq h$. As shown in \cite{KelLen09}, this limit always exists and $0\leq M_t(x)\leq 1$ for all $x\in V$.

The function $M_t$ consists of two parts which can be interpreted as follows: the first term, $e^{-tL}1(x)$, is the heat which is still in the graph at time $t$. The integral denotes the heat which was killed by the potential up to time $t$. Thus,  $1-M_t$ can be interpreted as the heat which is transported to the boundary of the graph.

In this setting, Theorem 1 of \cite{KelLen09} (see also Proposition 28 of \cite{KelLen10}) states the following:
\begin{proposition}\emph{(Characterization of stochastic completeness,  \cite[Theorem~1]{KelLen09})} \label{character}
Let $(V,b,c,m)$ be a weighted graph. The following statements are equivalent:
\begin{enumerate}
\item[(i)] There exists $v: V \to [0, \infty)$ non-zero, bounded such that  $( \Lf + \al ) v \leq 0$ for some (equivalently, all) $\al > 0.$
\item[(ii)] There exists $v: V \to (0, \infty)$ bounded such that  $( \Lf + \al ) v = 0$ for some (equivalently, all) $\al > 0.$
\item[(iii)] $M_t(x) < 1$ for some (equivalently, all)  $x \in V$ and $t > 0.$
\item[(iv)]  There exists a non-trivial, bounded solution to the heat equation with initial condition $u_0 \equiv 0.$
\end{enumerate}
\end{proposition}

Such weighted graphs are called \emph{stochastically incomplete at infinity}.  Otherwise,  a weighted graph is called \emph{stochastically complete at infinity}.  This extends the usual notion of stochastic completeness for the Laplacian to the case where a potential is present.  Clearly, as already discussed in \cite{KelLen09}, when the potential is zero, a stochastically complete weighted graph is also stochastically complete at infinity.

\smallskip

In order to formulate our stochastic completeness comparison theorems, we need to compare the potentials of two weighted graphs, continuing Definition~\ref{d:C_comp}.
\begin{definition}
We say that a weighted graph $(V, b,c,m)$ has \emph{stronger \tu{(respectively,} weaker\tu{)} potential} with respect to $x_0\in V$ than a  weakly spherically symmetric graph $(\s V, \s b, \s c,\s m)$ if, for all $x \in S_r(x_0) \subset V$ and $r \in \N_0$,
\[  \p(x) \geq {\s \p(r)} \qquad \left( \tu{respectively, }\; \p(x) \leq {\s \p(r)} \right). \]
\end{definition}

\smallskip
\subsection{Theorems and remarks}
It is desirable to have conditions which imply stochastic completeness or incompleteness at infinity and this is our goal.  We start with a characterization of stochastic completeness at infinity for weakly spherically symmetric graphs whose proof will be given at the end of the section. It generalizes a result for the graph Laplacian $\Lp$ on spherically symmetric graphs found in \cite{Woj10}.

\begin{theorem}\label{t:stochastic}\emph{(Geometric characterization of stochastic completeness)} A weakly spherically symmetric graph
$(\s V,\s b,\s c,\s m)$  is stochastically complete at infinity if and only if
\[ \sum_{r=0}^\infty \frac{V_{\p+1}(r)}{\partial  B(r)} = \infty \]
where
$$V_{\p+1}(r) =\sum_{x\in B_r }(\s \p(x)+1)\s m(x) \quad \mbox{and} \quad
\partial  B(r) = \s \ka_+(r)\s m(S_r).$$
\end{theorem}

Combining this with Theorem~\ref{t:spectrum} we get an immediate corollary, which is an analogue to theorems found in \cites{BP06, Ha09} for the Laplacian on a Riemannian manifold. The proof is given right after the proof of Theorem~\ref{t:stochastic}.

\begin{corollary}\label{c:stochastic}
If a weakly spherically symmetric graph
$(\s V,\s b,\s c,\s m)$  is stochastically incomplete at infinity, then
\[ \lm_0(\s L) > 0 \tu{ and } \si(\s L) = \sd(\s L). \]
\end{corollary}

\smallskip
\begin{remarks}
(a) The converse statements do not hold.  For example, in the unweighted case, both $\Lp$ and $\Ln$ on regular trees of degree greater than 2 have positive bottom of the spectrum but govern stochastically complete processes.  Furthermore, as shown in \cites{Ura99, Kel10}, for a tree one has $\si(\Lp) = \sd(\Lp)$ whenever the vertex degree goes to infinity along any sequence of vertices which eventually leaves every finite set, while stochastic incompleteness requires that the vertex degree goes to infinity at a certain rate as shown in \cites{Woj09} (see also Section \ref{s:Application}).

(b) The statements of the corollary do not hold for general  weighted graphs.  This can be seen from stability results for stochastic incompleteness at infinity proven in \cites{Woj09, KelLen09, Hu} which state that attaching any graph to a graph which is stochastically incomplete at infinity at a single vertex does not change the stochastic incompleteness.  Therefore, starting with a stochastically incomplete spherically symmetric tree, attachment of  a single path to infinity can  drive the bottom of the spectrum down to zero and add essential spectrum  (as follows by general principles) without effecting the stochastic incompleteness.

(c) The two statements of the corollary are not completely independent: if a Laplacian on a graph has purely discrete spectrum and the constant function $1$ does not belong to $\ell^2(V,m)$ or $c\not \equiv 0$, then the lowest eigenvalue cannot be zero.

\end{remarks}

\smallskip

Our second main result of this section is a comparison theorem for stochastic completeness in the spirit of \cite{Ich82}.

\begin{theorem}\label{t:sc_comparison}\emph{(Stochastic completeness at infinity comparison)}
If a weighted graph $(V,b,c,m)$ has stronger curvature growth and weaker potential (respectively, weaker  curvature growth and stronger potential) than a weakly spherically symmetric graph $(\s V,\s b,\s c,\s m)$ which is stochastically incomplete (respectively, complete) at infinity, then $(V,b,c,m)$
is stochastically incomplete (respectively, complete) at infinity.
\end{theorem}

\begin{remarks}

(a) Note that  a stronger potential can make a stochastically incomplete graph stochastically complete at infinity (compare with Theorem~2 in \cite{KelLen09}). This is due to the definition of $M_t$.  Specifically, the potential kills heat in the graph and, as such, prevents it from being transported to infinity.

(b) For the results above to hold, it suffices that the comparisons hold outside of a finite set. This is due to the fact that stochastic (in)completeness is stable under finite dimensional perturbations, compare \cite{Woj09, KelLen09, Hu}.
\end{remarks}

\smallskip
\subsection{Proofs of Theorems \ref{t:stochastic} and \ref{t:sc_comparison}}
We begin with an observation concerning the  solutions to the difference equation on weakly spherically symmetric graphs which we have encountered before.

\begin{lemma}\label{l:bounded}\emph{(Boundedness of spherically symmetric solutions)}
Let a weakly spherically symmetric graph $(\s V,\s b,\s c,\s m)$ be given.  Let $\al>0$ and  $v:\s V\to(0,\infty)$ be a spherically symmetric function such that $(\s{\ow L}+\al)v=0$. Then, $v$ is unbounded if and only if
\begin{align*}
\sum_{r=0}^\infty \frac{V_{\p+1}(r)}{\partial  B(r)} = \infty.
\end{align*}
\end{lemma}
\begin{proof} We will use the obvious fact that
$\sum_{r=0}^{\infty} \frac{V_{\p+1}(r)}{\partial  B(r)} = \infty$
if and only if $\sum_{r=0}^\infty \frac{V_{\p+ \alpha}(r)}{\partial  B(r)} = \infty$ for some (equivalently, all)  $\alpha >0$.

By the recursion formula of Lemma~\ref{radialsolutions}, the function $v$ satisfies
\begin{equation*}\label{recursion3}
 v(r+1) - v(r) = \frac{1}{\partial  B(r)} \sum_{i=0}^r C_{\p+ \al }(i) v(i)
\end{equation*}
where $C_{\p+\al}(r) = (\s \p(r) + \al) \s m(S_r)$ and consequently is monotonously increasing as $v(0)>0$.   Hence, it satisfies
\[ v(r+1) - v(r) \geq \frac{V_{\p+\al }(r)}{\partial  B(r)} v(0) \]
where $V_{\p+\al}(r) =  \sum_{x \in B_r} (\s q (x)  + \alpha ) \s m(x) = \sum_{i=0}^r C_{\p+\al}(i)$.
Therefore, if $\sum_{r=0}^\infty \frac{V_{\p+1}(r)}{\partial  B(r)} = \infty$, then $v(r)=\sum_{i=0}^{r-1} ( v(i+1) - v(i)) \to \infty$ as $r\to\infty$. Hence, $v$ is unbounded in this case.

On the other hand, the recursion formula and monotonicity imply that
\[ v(r+1) \leq \left( 1 + \frac{V_{q+\al}(r)}{\partial  B(r) } \right) v(r) \leq \prod_{i=0}^r \left( 1 + \frac{ V_{q+\al}(i) }{\partial  B(i) } \right) v(0). \]
Therefore, if $\sum_{r=0}^\infty \frac{V_{\p+1}(r)}{\partial  B(r)} < \infty$, then $\prod_{r=0}^\infty \left( 1 + \frac{V_{\p+\al}(r)}{\partial  B(r)} \right) < \infty$ so that $v$ is bounded.
\end{proof}

We combine the lemma above with the characterizations of Proposition~\ref{character} in order to prove stochastic incompleteness at infinity of weakly spherically symmetric graphs. For the other direction, we will need the following criterion for stochastic completeness at infinity which is an analogue for a criterion of Has{\cprime}minski{\u\i} from the continuous setting \cite{Has60}.  Recently, Huang \cite{Hu} has proven a slightly stronger version in the case $c\equiv0$. %However, the proofs can be carried over directly. %As we allow for non-negative killing terms, we give an alternative proof for (not necessarily locally finite) weighted graphs in the appendix, see Theorem~\ref{t:Has2}. The special case of locally finite graphs is stated below.
For $v:V\to\R$, we write
\begin{align*}
    v(x)\to\infty\quad\mbox{as }x\to\infty
\end{align*}
whenever for every $C\geq 0$ there is a finite set $K$ such that $v\vert_{V\setminus K}\geq C$.

\begin{proposition}\label{t:Has}\emph{(Condition for stochastic completeness)}
If on a weighted graph $(V,b,c,m)$  there exists $v$ such that
$(\Lf + \al) v \geq 0$
 and
$ v(x) \to \infty \tu{ as } x \to \infty,$
then $(V,b,c,m)$ is stochastically complete at infinity.
\end{proposition}
\begin{proof}
Suppose there is a function $0\leq w\leq 1$ that solves $(\ow L+\al)w=0$. For given $C>0$ let $K\subset V$ be a finite set such that $v\vert_{V\setminus K}\geq C$. Then $u=v-Cw$ satisfies $(\ow L+\al)u\geq0$ on $V$ and $u\geq0$ on $V\setminus K$. As $K$ is finite, the negative part of $u$ attains its minimum on $K$. Therefore, by a minimum principle, see \cite[Theorem~8]{KelLen09}, $u\geq0$ on $K$. Hence, $v\geq Cw$ for all $C>0$. This implies that $w\equiv0$. By Proposition~\ref{character}, stochastic completeness at infinity follows.
\end{proof}
%\begin{proof} See Theorem~\ref{t:Has2} in the appendix.\end{proof}

\begin{proof}[Proof of Theorem~\ref{t:stochastic}]
By Lemma~\ref{radialsolutions}, there is always a strictly positive, spherically symmetric solution to $(\widetilde{L}^{\tu{sym}} + \al) v = 0$ for $\al>0$. By Lemma~\ref{l:bounded}, this solution is bounded if and only if the sum in the statement of  Theorem~\ref{t:stochastic} converges. In the case of convergence, we conclude stochastic incompleteness at infinity by Proposition~\ref{character}. On the other hand, if the sum diverges, the solution is unbounded and satisfies the assumptions of Proposition~\ref{t:Has} (by the spherical symmetry). Hence, the graph is stochastically complete at infinity.
\end{proof}

\begin{proof}[Proof of Corollary~\ref{c:stochastic}]
If the graph $(\s V,\s b,\s c,\s m)$ is stochastically incomplete at infinity, then the sum of Theorem~\ref{t:stochastic} converges. As a consequence, the sum of Theorem~\ref{t:spectrum} converges which implies positive bottom of the spectrum and discreteness of the spectrum of the graph $(\s V,\s b,\s m)$. Now, a non-negative potential only lifts the bottom of the spectrum (as can be seen from the Rayleigh-Ritz quotient) which gives the statement.
\end{proof}

\begin{proof}[Proof of Theorem \ref{t:sc_comparison}]
Let $(\s V,\s b,\s c,\s m)$ be a weakly spherically symmetric graph.  Let $v$ be the spherically symmetric solution with $ v(0) =1$  given by Lemma~\ref{radialsolutions} for $\al>0$. Then, $v$ is strictly positive and  monotonously increasing. Define $w:V\to(0,\infty)$ on the weighted graph $(V,b,c,m)$ by $w(x)=v(r)$ for $x\in S_r(x_0)\subset V$.

First, assume the stochastic incompleteness at infinity of  $(\s V,\s b,\s c,\s m)$.
Then, by Lemma~\ref{l:bounded} combined with Theorem~\ref{t:stochastic}, the function $v$, and thus $w$, is bounded. Under the assumptions that $(V,b,c,m)$ has stronger curvature growth and weaker potential than  $(\s V,\s b,\s c,\s m)$ and, as $v$ is monotonously increasing by Lemma~\ref{radialsolutions}, the function  $w$ satisfies for $x \in S_{r}$, $r \in \N$,
\begin{align*}
\big( \Lf + \al \big)& w(x)\\
&= \ka_{+}(x) \big( v(r) - v(r+1) \big) + \ka_{-}(x) \big( v(r) - v(r-1) \big)  + \left(\p(x) + \al \right) v(r)\\&
\leq \big( \widetilde{L}^{\tu{sym}}  + \al \big) v(r) = 0.
\end{align*}
Hence, the graph $(V,b,c,m)$ is stochastically incomplete at infinity,  by Proposition~\ref{character}.

Assume now that  $(\s V,\s b,\s c,\s m)$ is stochastically complete at infinity. Then, again by Proposition~\ref{character}, the function $v$ and, thus $w$, is unbounded. As above, under the assumptions of weaker curvature growth and stronger potential, one checks that
\[ \big( \Lf + \al \big) w(x) \geq \big(  \widetilde{L}^{\tu{sym}}  + \al \big) v(r) = 0 \]
for all $x\in S_{r}\subset V$, $r \in \N_0$. Since $w$ is unbounded, $w(x) \to \infty$ as $x \to \infty$, therefore, the weighted graph $(V,b,c,m)$ is stochastically complete at infinity by the condition for stochastic completeness, Proposition~\ref{t:Has}.
\end{proof}

\section{Applications to graph Laplacians}\label{s:Application}
In this section we  discuss the results of this paper for unweighted graphs. Thus, we consider the situation  $b:V \times V \to \{ 0,1 \}$ and  $c\equiv0$. Chosing $m\equiv 1$ and $m \equiv \deg$ we obtain the two graph Laplacians, that is, $\Lp$ on $\ell^2(V)= \ell^2(V,1)$  and $\Ln$ on $\ell^{2}(V,\deg)$.

Note that, in the setting for $\Lp$, the curvatures $\ka_\pm :V \to \N$ denote the number of edges connecting a vertex $x$, which lies in the sphere of radius $r = d(x,x_0)$, to vertices in the spheres of radius $r\pm1$.  Therefore, we can think of  $\ka_{-}$ and $\ka_{+}$ as the inner and outer vertex degrees, respectively. On the other hand, for $\Ln$, the curvatures $\oh\ka_\pm :V \to \Q$ are inner and outer vertex degree divided by the degree.

\bigskip

In this situation, we have the following corollaries of Theorem~\ref{t:HK_wsp} for the heat kernels.

\begin{corollary}The operator $\Lp$ has a spherically symmetric heat kernel if and only if the inner and outer vertex degrees are spherically symmetric.
\end{corollary}

\begin{corollary}The operator $\Ln$ has a spherically symmetric heat kernel if and only if the ratio of  inner and outer vertex degrees and the vertex degree are spherically symmetric.
\end{corollary}

Note that, in order to have a spherically symmetric heat kernel for $\Ln$, even less symmetry than the graphs of Figure~\ref{f:wsp_vs_sp} possess is needed.

\medskip

Next, we fix $m \equiv 1$ to focus on $\Lp$ and give examples of weakly spherically symmetric graphs which satisfy our summability condition.
%  In particular, we will give examples of graphs of polynomial volume growth such that $\Lp$ has discrete spectrum with the bottom of the spectrum positive and which are stochastically incomplete. This is surprising, as it shows that there are no direct analogues to the results of Brooks \cite{Br81} and Grigor{\cprime}yan \cite{Gri99} on manifolds for $\Lp$ with respect to the natural graph distance.
We start with the case of unweighted spherically symmetric trees.  Here, $\ka_-(r)=1$ and, if $V(r):= V_1(r)$, the summability criterion of Theorems~\ref{t:spectrum} and \ref{t:stochastic} concerns the convergence or divergence of
$$ \sum_{r=0}^\infty\frac{V(r)}{\partial  B(r)} =\sum_{r=0}^\infty\frac{|B_{r}|}{|S_{r+1}|}$$
%where $V(r):= V_1(r)$.
%For  such unweighted  spherically symmetric trees we have
%$$  |S_{r+ 1}| = \partial B(r)=\ka_{+}(r)|S_{r}|.$$
%Here, the first equality holds for arbitrary trees and the second equality is true for all weakly spherically symmetric graphs with $m \equiv 1$.
which can be seen to be equivalent to the convergence or divergence of
$\sum_{r=0}^\infty \frac{1}{\ka_+(r)}$,
see \cite{Woj10}.
%As noted in \cite{Woj10}, for trees, the convergence of the sum in $(*)$  is equivalent to the convergence of the sum $\sum_{r=0}^\infty \frac{1}{\ka_+(r)}$.  This can be seen from the fact that for trees $|S_{r+1}| = \prod_{i=0}^r \ka_+(i)$ so that
%\[ \frac{|B_r|}{|S_{r+1}|} = \frac{1+ \sum_{i=0}^{r-1} \prod_{j=0}^i \ka_+(j)}{\prod_{j=0}^r \ka_+(j)}. \]
%The stated equivalence then follows by the limit comparison test for series.

Using this, it follows easily, as was already shown in \cite{Woj09}, that the threshold for the volume growth for stochastic completeness of spherically symmetric trees lies at $r!\sim e^{r\log r}$.  This already stands in contrast to the result of Grigor{\cprime}yan which puts the threshold for stochastic completeness of manifolds at $e^{r^{2}}$ \cite{Gri99}. However, this is still in line with the result of Brooks which yields that subexponential growth implies absence of a spectral gap for manifolds \cite{Br81}.

\bigskip

We now come to the family of spherically symmetric graphs, here called antitrees, which yield our surprising examples.
\begin{definition} An unweighted weakly spherically symmetric graph with $m \equiv 1$ is called an \emph{antitree} if $\ka_{+}(r)=|S_{r+1}|$ for all $r\in\N_0$.
\end{definition}

As opposed to trees, which are connected graphs with as few connections as possible between spheres, antitrees have the maximal number of connections.  The contrast is illustrated in Figure~\ref{f:tree_vs_antitree}.  Such graphs have already been used as examples in \cite{DoKa88,Woj10,Web10}.

\begin{center}
\begin{figure}[h]
%\vspace{}
\includegraphics[scale=0.5]{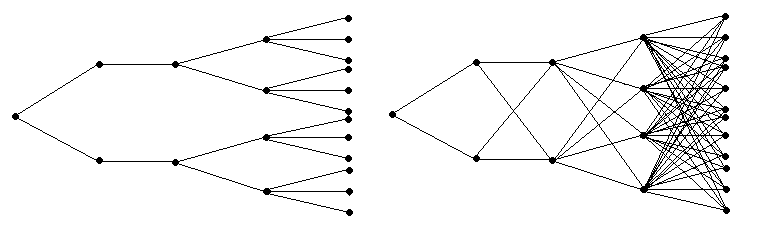}
\caption{A tree and an antitree}
\label{f:tree_vs_antitree}
\end{figure}
\end{center}

In the case of antitrees, the summability criterion concerns the convergence or divergence of $$\sum_{r=0}^\infty\frac{V(r)}{\partial  B(r)} =\sum_{r=0}^\infty\frac{|B_{r}|}{|S_{r}||S_{r+1}|}.$$
For functions $f,g:\N_0\to\R$ we write $f \sim g$ provided that there exist constants $c$ and $C$ such that $c g(r) \leq f(r) \leq C g(r)$ for all $r\in \N_0$.  We can then relate the summability criterion above to volume growth through the following lemma whose proof is immediate.

\begin{lemma}\label{l:antitrees} \emph{(Volume growth of antitrees)} Let $\be>0$.  An antitree with
$\ka_{+}(r)\sim r^{\be}$ for $r\in \N_0$ satisfies $V(r)\sim  r^{\be+1}$.  Furthermore, if $\be> 2$ (respectively, $\be \leq2$), then
\begin{align*}
\sum_{r=0}^\infty\frac{V(r)}{\partial  B(r)}<\infty\quad \bigg(\mbox{respectively, } \sum_{r=0}^\infty\frac{V(r)}{\partial  B(r)}=\infty \bigg).
\end{align*}
\end{lemma}
%\begin{proof}
%By the definition of antitrees, for $\be>0$ and $\ka_{+}(r)\sim r^\be$, which is equivalent to $|S_r| \sim r^\be$, we have
%\begin{align*}
%V(r) \geq c\sum_{n=0}^{r}n^{\be}\geq c\sum_{n\geq r/2}^{r}n^{\be}\geq c\frac{1 {2^{\be+1}}r^{\be+1}.
%\end{align*}
%Furthermore, it is clear that $V(r)\leq C r^{\be+1}$.  This proves the first statement.

%Let us now turn to the divergence and convergence of the infinite sums. For $\be>0$ and $\ka_{+}(r)\sim r^{\be}$, we have by the above that $V(r)\sim r^{\be+1}$. Hence, for $r\geq1$, we have
%\begin{align*}
%s(r)=\sum_{j=0}^r\frac{V(j)}{\partial  B(j)} = \sum_{j=0}^r\frac{V(j)}{\ka_{+}(j-1)\ka_{+}(j)} \sim \sum_{j=0}^r\frac{1}{j^{\be-1}}.
%\end{align*}
%Clearly, $s(r)\to\infty$, as $r\to\infty$, if $\be\leq2$ and is bounded for $\be>2$.
%\end{proof}

Combining this lemma with Theorem~\ref{t:stochastic} immediately gives the following  corollary.

\begin{corollary}\emph{(Polynomial growth and stochastic incompleteness)} Let $\be>0$. An antitree with  $\ka_{+}(r)\sim r^{\be}$  satisfies  $V(r)\sim r^{\be+1}$ and is stochastically complete if and only if $\be\leq 2$.
\end{corollary}

This phenomenon was already observed in \cite{Woj10} and it gives an even sharper contrast to Grigor{\cprime}yan's threshold for manifolds as the volume growth does not even have to be exponential for the antitree to become stochastically incomplete.  Furthermore, in \cite[Theorem 1.4]{GHM}, it is shown that any graph whose volume growth is less than cubic is stochastically complete.  Thus, our examples are, in some sense, the stochastically incomplete graphs with the  smallest volume growth.

It should also be mentioned that it was recently shown by Huang \cite{Hu},  using ideas found in \cite{BB10}, that the condition $\sum_{r=0}^\infty\frac{V(r)}{\partial  B(r)}= \infty$ which implies stochastic completeness of weakly spherically symmetric graphs does not imply stochastic completeness for general graphs.

\bigskip
Let us now turn to a discussion of the spectral consequences. Combining the volume growth of antitrees, Lemma~\ref{l:antitrees},  with Theorem~\ref{t:spectrum} gives the following immediate corollary.

\begin{corollary}\label{c:polynom_spectrum}\emph{(Polynomial growth and positive bottom of the spectrum)} Let $\be>2$. An antitree with  $\ka_{+}(r)\sim r^{\be}$  satisfies  $V(r)\sim r^{\be+1}$ with
\begin{align*}
\lm_0(\Lp)>0\quad\mbox{and}\quad\si(\Lp)=\sd(\Lp).
\end{align*}
\end{corollary}

This corollary shows that, for $\Lp$, there are no direct analogues to Brook's theorem which states, in particular, that subexponential volume growth of manifolds implies that the bottom of the essential spectrum is zero. On the other hand, there is a result by Fujiwara \cite{Fuj96a} (which was later generalized in \cite{Hi03}) for the normalized Laplacian $\Ln$ acting on $\ell^2(V, \deg)$, which states that
\begin{equation*}\label{volume}
\lmess(\Ln) \leq  1 - \frac{2 e^{\frac{\mu}{2}}}{1 + e^\mu}.
\end{equation*}
Here,  the exponential volume growth is given by $\mu = \limsup_{r \to \infty} \frac{1}{r}\log\oh V_{1}(r)$ where $\oh V_{1}\ab{r}= \sum_{x\in B_r}\deg(x)$.  Therefore, if a graph has subexponential volume growth with respect to the measure $\deg$, it follows that $\lm_0(\Ln)\leq\lmess(\Ln)=0$.

\bigskip
Let us also discuss  how our Theorem~\ref{t:spectrum} complements some of the results found in \cite{Kel10,KelLen10}.  There it is shown that, for graphs with positive Cheeger constant at infinity, $\al_\infty > 0$, rapid branching is equivalent to discreteness of the spectrum of the Laplacian.   The constant $\al_\infty$ is defined as the limit over the net of finite sets $K$ of the quantities
\[ \al_{K}=\inf \frac{\partial W}{m(W)} \]
where $\partial  W$ is the number of edges leaving $W$ and the infimum is taken over all finite sets $W\subseteq V\setminus K$, see \cite{Fuj96b,Kel10,KelLen10}.  Now, subexponential volume growth, that is, $\mu=0$ implies $\al_{\infty}=0$ by $1-\sqrt{1-\al_{\infty}^{2}}\leq\lmess(\Ln)$, shown in \cite{Fuj96b}, combined with the estimate for $\lmess(\Ln)$ given above.
Corollary~\ref{c:polynom_spectrum} gives  examples of  graphs with $\mu=0$ and thus $\al_{\infty}=0$ but for which $\Lp$ has no essential spectrum.

\bigskip
Another interesting consequence of Theorem~\ref{t:HK_comparison} and Theorem~\ref{t:spectral_comparison} for $\Lp$ is the following:
\begin{corollary}\label{c:edges} Let $\Lp_G$ and $\Lp_{G'}$ be the graph Laplacians of two unweighted weakly spherically symmetric graphs which have the same curvature growth. Then, $\lm_0(\Lp_G)=\lm_0(\Lp_{G'})$.
\end{corollary}

This means, in particular, that the Laplacians $\Lp$ on all graphs in Figure~\ref{f:wsp_vs_sp} have the same bottom of the spectrum.  Note that this is not at all true  for the normalized graph Laplacian $\Ln$ as it operates on $\ell^2(V,m) = \ell^2(V,\deg)$ and the presence of edges connecting vertices on the same sphere clearly effects the degree measure.

To illustrate this contrast, note that on a $k$-regular tree $T_k$ the bottom of the spectrum is known to be $\lm_0(\Ln_{{T}_k})=\frac{1}{k}(k - 2 \sqrt{k-1})$ for all $k$. On the other hand, if one connects all vertices in each sphere one obtains a graph $G_k$ such that $\lm_0(\Ln_{G_k})=0$  as shown in \cite[Theorem 6]{Kel10}.   However, from  Corollary \ref{c:edges} above, $\lm_0(\Lp_{{T}_k})=\lm_0(\Lp_{G_k})=k - 2 \sqrt{k-1}$ which is also new compared to   \cite[Theorem 6]{Kel10}, where only $\lm_0(\Lp_{G_k})\leq k$ is shown.  In particular, we have another example where the ground state energies of $\Lp$ and $\Ln$ differ.

\bigskip

\appendix\section{Reducing subspaces and commuting operators}\label{s:Appendix}
We study symmetries of  selfadjoint operators. These symmetries are given in terms of bounded operators commuting with the selfadjoint operator in question. We present a general characterization in Theorem \ref{abstract-characterization}.  With Lemma \ref{Symmetry-Friedrich}, we then turn to the question of how symmetries of a symmetric non-negative operator carry over to its Friedrichs extension. Finally, we specialize to the situation in which the bounded operator is a projection onto a closed subspace. The main result of this appendix, Corollary \ref{main-appendix},
 characterizes when a selfadjoint operator commutes with such a projection.

While these  results are certainly known in one form or another, we have not found all of them  in the literature in the form discussed below. In the main body of the paper they will be applied to Laplacians on graphs. However, they are general enough to be applied to Laplace-Beltrami operators on manifolds as well.

\bigskip

A subspace $U$ of a Hilbert space is said to be \emph{invariant} under the bounded operator $A$ if $A$ maps $U$ into $U$.

\begin{thm} \label{abstract-characterization} Let $L$ be a selfadjoint non-negative  operator on the Hilbert space $\mathcal{H}$  and $A$ a bounded operator on $\mathcal{H}$. Then, the following assertions are equivalent:

\begin{itemize}
\item[(i)]  $D(L)$ is invariant under $A$ and    $L A x= A Lx $ for all $x\in D(L)$.

\item[(ii)] $D(L^{1/2})$  is invariant under $A $ and  $L^{1/2} A x= A L^{1/2}x$ for all $x\in D(L^{1/2})$.

\item[(iii)] $1_{[0,t]} (L)  A = A 1_{[0,t]} (L)$ for all $t\geq 0$.

\item[(iv)]  $e^{- t L} A = A e^{- t L} $ for all $t\geq 0$.

\item[(v)]  $ (L + \alpha)^{-1} A = A (L+ \alpha)^{-1}$  for all $\alpha >0$.

\item[(vi)] $ g(L) A = A g(L)$ for all bounded measurable $g : [0,\infty) \longrightarrow \C$.

\end{itemize}

\end{thm}
\begin{proof} This is essentially  standard. We sketch a proof for the convenience of the reader.
We first show that (iii), (iv), (v) and (vi) are all equivalent:

\smallskip

(iii)$\Longrightarrow$(iv): This follows by a simple approximation argument.

\smallskip

(iv)$\Longrightarrow$(v): This follows immediately from  $(L + \alpha)^{-1} = \int_0^\infty e^{- t \alpha} e^{- t L } dt$ (which, in turn, is a direct consequence of the spectral calculus).

\smallskip

(v)$\Longrightarrow$(vi):  The assumption (v) together with a  Stone/Weierstrass-type argument shows that $ g(L) A = A g(L)$ for all continuous $g : [0,\infty)\longrightarrow \C$ with $g(x) \to 0$ for $x\to \infty$. Now, it is not hard to see that the set
$$\{ f: [0,\infty) \longrightarrow \C \ | \ \mbox{$f$ measurable and bounded with $f(L) A = A f(L)$}\}$$
is closed under pointwise convergence of uniformly bounded sequences. This gives the desired statement (vi).

\smallskip

(vi)$\Longrightarrow$(iii): This is obvious.

\smallskip

We now show (ii)$\Longrightarrow$(i)$\Longrightarrow$(v) and  (vi)$\Longrightarrow$(ii).

\smallskip

(ii)$\Longrightarrow$(i): This  is clear as $L = L^{1/2} L^{1/2}$.

\smallskip

(i)$\Longrightarrow$(v):   Obviously, (i) implies $A (L + \alpha) x = (L  + \alpha) A x $ for all $\alpha\in \R$ and $x\in D(L)$.
As $(L+\alpha )$ is injective   for $\alpha >0$ we infer for all such $\alpha $ that
$$(L + \alpha)^{-1} A  = A  (L+ \alpha)^{-1}.$$

(vi)$\Longrightarrow$(ii): For every natural number $n$ the operator $L^{1/2} 1_{[0,n]} (L)   = (id^{1/2} 1_{[0,n]}) (L)$ is a bounded operator commuting with $A$ by (vi).  Let $x\in D(L^{1/2})$ be given and set  $x_n :=1_{[0,n]} (L) x$.  Then, $x_n$ belongs to  $D(L^{1/2})$. Moreover, as $ 1_{[0,n]} (L)$ is a projection, we obtain  by (vi)  that
$$ A  x_n = A  1_{[0,n]} (L)   x =   1_{[0,n]} (L) A 1_{[0,n]} (L) x =1_{[0,n]} (L) A x_n. $$
In particular, $A x_n$ belongs to $D(L^{1/2})$ as well.
This gives, by (vi) again, that
$$ L^{1/2}  A  x_n = L^{1/2} 1_{[0,n]} (L)  A  x_n = A L^{1/2} 1_{[0,n]}(L) x.$$
As $x$ belongs to  $D(L^{1/2})$, we infer that $ L^{1/2} 1_{[0,n]} (L) x$ converges to $L^{1/2} x$. Moreover, $A x_n$ obviously converges to $A x$. As $L^{1/2}$ is closed, we obtain that $A x$ belongs to $D(L^{1/2})$ as well and
  $L^{1/2}  A x = L^{1/2} A x$ holds.
  \end{proof}

\begin{remark}
(a) The method to prove (v)$\Longrightarrow$(vi) can be strengthened as follows:  Let $L$ be a selfadjoint operator with spectrum $\Sigma$.  Let $\mathcal{B}(\Sigma)$ be the algebra of all bounded measurable functions on $\Sigma.$ A sequence $(f_n)$ in $\mathcal{B}(\Sigma)$ is said to converge to $f\in \mathcal{B}(\Sigma)$ in the sense of $(\clubsuit)$ if the $(f_n)$ are uniformly bounded and converge pointwise to $f$. Let  $F$  be a subset of $\mathcal{B}$ such that  $f(L) A  = A  f(L)$ holds for all $f\in F$. If  the smallest subalgebra of $\mathcal{B}$  which contains $F$ and is closed under convergence with respect to $(\clubsuit)$ is $\mathcal{B}$, then  $g(L) A =  A g(L)$ for all $g\in \mathcal{B}$.

(b)  If $L$ is  an arbitrary selfadjoint operator then the equivalence of (i), (iii) and (vi) is still true and the semigroup in (iv) can be replaced by the unitary group and the resolvents in (v) can be replaced by  resolvents for $\alpha \in \C\setminus \R$ (as can easily be seen using (a) of this remark).
\end{remark}

\bigskip

\begin{definition} Let $L$ be a selfadjoint non-negative operator on a Hilbert space $\mathcal{H}$ and $A $ a bounded operator on $\mathcal{H}$. Then, $A$ is said to \emph{commute} with $L$ if one of the equivalent statements of the theorem holds.
\end{definition}

\begin{corollary} Let $L$ be a selfadjoint non-negative operator on a Hilbert space $\mathcal{H}$ and $A $ a bounded operator on $\mathcal{H}$. Then, $A$ commutes with $L$ if and only if its adjoint  $A^\ast$ commutes with $L$.
\end{corollary}
\begin{proof} Take adjoints in (iii) of the previous theorem.
\end{proof}

A simple situation in which the previous theorem can be applied is given next.

\begin{proposition} \label{useful}  Let $L$ be a selfadjoint non-negative  operator on the Hilbert space  $\mathcal{H}$ and let $A$ be a bounded operator on $\mathcal{H}$. Let, for each natural number  $n$,  a closed  subspace  $\mathcal{H}_n$ of $\mathcal{H}$ be given with $A \mathcal{H}_n \subset \mathcal{H}_n$ and $\overline{\cup_{n} \mathcal{H}_n} = \mathcal{H}$.  If, for each $n$, there exists  a selfadjoint non-negative  operator $L_n$ from $\mathcal{H}_n$ to $\mathcal{H}_n$ with $ A L_n = L_n A$ and
$$(L_{n+ k} + \alpha)^{-1} x  \to (L  + \alpha)^{-1} x, \;k\to \infty,$$
for all natural numbers $n$,  $x\in \mathcal{H}_n$ and $\alpha >0$, then $A L = LA$ holds. A corresponding statement holds with resolvents  replaced by the semigroup.
\end{proposition}
\begin{proof}  By assumption we have $A (L + \alpha)^{-1}x = (L + \alpha)^{-1} A x$ for all $x$ from the dense set   $\cup_{n} \mathcal{H}_n$. By boundedness of the respective operators we infer $A (L + \alpha)^{-1}  = (L + \alpha)^{-1} A$  and the statement follows from the previous theorem.
\end{proof}

\bigskip

The previous theorem deals with symmetries of a selfadjoint operator $L$.  Often, the selfadjoint operator arises as the Friedrichs extension of a symmetric operator.
We next study how symmetries of a symmetric operator carry over to its  Friedrichs extension. Specifically,  we consider  the following situation:

\begin{itemize}
\item[$(*)$] Let $\mathcal{H}$ be a Hilbert space with inner product $\langle\cdot, \cdot\rangle$. Let $L_0$ be a symmetric operator on $\mathcal{H}$ with domain $D_0$. Let $Q_0$ be the associated form, i.e., $Q_0$ is defined on $D_0 \times D_0$ via
    $Q_0 (u,v) :=\langle L_0 u, v\rangle.$
    Assume that $Q_0$ is non-negative, i.e., $Q_0 (u,u)\geq 0$ for all $u\in D_0$. Then, $Q_0$ is closable. Let $Q$ be the closure of $Q_0$, $D(Q)$ the domain of $Q$ and  $L$  the Friedrichs extension of $L_0$, i.e., $L$ is the  selfadjoint operator associated to $Q$.
\end{itemize}

\begin{lemma} \label{Symmetry-Friedrich} Assume $(*)$.  Let $A$ be a bounded operator on $\mathcal{H}$ with $D_0$ invariant under $A$ and $A^\ast$,  $A L_0 x = L_0 A x$ and $A^\ast L_0 x = L_0 A^\ast x$ for all $x\in D_0$. Then, the following assertions are equivalent:

\begin{itemize}

\item[(i)] $D(L)$ is invariant under $A$ and  $A L x  = L A x $ for all $x\in D(L)$.

\item[(ii)] $D(Q)$ is invariant under $A$  and $A^\ast$ and $Q (A x, y) = Q (x, A^\ast y)$ for all $x,y\in D(Q)$.

\item[(iii)] There exists a $C\geq 0$ with  both $Q_0 (A x, Ax) \leq C Q_0 (x,x)$  and $Q_0 (A^\ast x, A^\ast x) \leq C Q_0(x,x) $ for all $x\in D_0$.

\end{itemize}

\end{lemma}
\begin{proof} (iii)$\Longrightarrow$(ii): By $A L_0 x = L_0 A x$ for all $x\in D_0$ we infer that $Q_0 (Ax, y) = Q_0 (x, A^\ast  y)$ for all $x,y\in D_0$.
 As $Q$ is the closure of $Q_0$, it now suffices to show that both $(A u_n)$  and $(A^\ast u_n)$ are  a Cauchy sequences with respect to the $Q$-norm, whenever $(u_n)$ is a Cauchy sequence with respect to the $Q$-norm in $D_0$. This follows directly from (iii).

 \smallskip

(ii)$\Longrightarrow$(i):  Let $x \in D(L)$ be given. Then, $x$ belongs to $D(Q)$  and, by (ii),  $Ax $ belongs to $D(Q)$ as well. Thus, we can calculate for all $y\in D(Q)$
$$Q (A x, y) = Q (x, A^\ast y) = \langle L x, A^\ast  y\rangle = \langle A L x, y\rangle.$$
This implies that $A x\in D(L)$ and $L A x = A L x$. Hence, we obtain (i).

\smallskip

(i)$\Longrightarrow$(iii): From Theorem \ref{abstract-characterization} and (i) we infer that $L^{1/2} A x = A L^{1/2}x$ for all $x\in D(L^{1/2})$.  Now,  for $x\in D_0$   it holds that
$$Q_0 (x,x) = \langle L_0 x, x\rangle = \langle L^{1/2} x, L^{1/2} x\rangle = \|L^{1/2} x\|^2.$$
By  $A x \in D_0$ for $x\in D_0$   a direct calculation gives
$$ Q_0 (A x, A x)  = \| L^{1/2} A x\|^2   =   \| A L^{1/2} x\|^2 \leq  \|A\|^2 \|L^{1/2} x\|^2 =  \|A\|^2 Q_0 (x,x).$$
A similar argument shows $Q_0 (A^\ast x, A^\ast x) \leq \|A^\ast\|^2 Q_0 (x,x)$.
This finishes the proof.
\end{proof}

We now turn to the special situation that $A$ is the projection onto a closed subspace. In this case, some further strengthening of the above result is possible. We first provide an appropriate definition.

\begin{definition} Let $\mathcal{H}$ be a Hilbert space and $S$ a symmetric operator on $\mathcal{H}$ with domain $D(S)$. A closed subspace $ U$ of $\mathcal{H}$  with associated orthogonal projection $P$ is called a \emph{reducing subspace} for $S$ if $D(S)$ is invariant under $P$ and $ S P  x = P S P x$ for all $x\in D(S)$.
\end{definition}

The previous definition is just a commutation condition in the form discussed above as shown in the next lemma.

\begin{lemma} Let $S$ be a symmetric operator on the Hilbert space $\mathcal{H}$ and $P$ be the orthogonal projection onto a closed subspace $U$ of $\mathcal{H}$. Then, the following assertions are equivalent:
\begin{itemize}
\item[(i)] $U$ is a reducing subspace for $S$.

\item[(ii)] $D(S)$ is invariant under $P$ and  $S P x = P S x $ holds for all $x\in D(S)$.
\end{itemize}
\end{lemma}
\begin{proof} The implication (ii)$\Longrightarrow$(i) is obvious. It remains to show (i)$\Longrightarrow$(ii):  We first show  $ P S y = 0$ for all $y\in D(S)$ with $P y = 0$ (i.e., $y \perp  U$): Choose $x\in D(S)$ arbitrarily. Then, as $Px \in D(S) \subset D(S^\ast)$ we obtain
$$ \langle P S y, x\rangle = \langle Sy, Px\rangle = \langle y, S^\ast  P x\rangle = \langle y, S Px\rangle  = \langle y, P S P x\rangle = \langle Py, S P  x\rangle = 0.$$
As $D(S)$ is dense, we infer $P S y=0$. Let now $x\in D(S)$ be arbitrary. Then, $x = Px + (1-P) x$ and both $Px$ and $(1-P) x$ belong to $D(S)$. Thus, we can calculate
$$ P S x = P  S Px +  P S (1-P) x = PS P x = S P x.$$
This finishes the proof.
\end{proof}

We now come to the main result of the appendix dealing with symmetries of symmetric operators in terms of reducing subspaces.

\begin{corollary} \label{main-appendix}  Assume $(*)$. Let $ U $ be a closed subspace of $\mathcal{H}$ and $A$ the orthogonal projection onto $ U $.  Assume that $D_0$ is invariant under $A$.  Then, the following assertions are equivalent:

\begin{itemize}
\item[(i)] $U$ is a reducing subspace for $ L_0$, i.e., $ L_0 A x = A L_0 x$ for all $x\in D_0$.

\item[(ii)]   $Q_0 (A x, A y) = Q_0 (A x, y) = Q_0 (x, Ay) $ for all $x,y\in D_0$.

\item[(iii)] $D(Q)$ is invariant under $A$   and $Q (A x, A y) = Q (Ax, y) = Q (x, Ay) $ for all $x,y\in D(Q)$.

\item[(iv)] $ U $ is a reducing subspace for $L$.

\item[(v)] $A$ commutes with $ e^{-t L}$ for every $t\geq 0$.

\item[(vi)] $A$ commutes with $(L+\alpha)^{-1}$ for any $\alpha >0$.

\end{itemize}

\end{corollary}

\begin{proof}
Obviously, (i) and (ii) are equivalent. The equivalence of (iii) and  (iv) follows from  the equivalence of (i) and (ii) in Lemma \ref{Symmetry-Friedrich}. The equivalence between (iv), (v) and (vi) follows immediately from Theorem \ref{abstract-characterization}. The implication (iii)$\Longrightarrow$(ii) is clear (as $A D_0 \subseteq D_0$). It remains to show (ii)$\Longrightarrow$(iii):  A direct calculation using (ii) gives for all $x\in D_0$ that
\begin{eqnarray*}
Q_0 (x,x)&=& Q_0 ((A + (1-A))x, x)\\
&=& Q_0 (Ax,x) + Q_0 ((1-A)x ,x)\\
&=& Q_0 (Ax,Ax) + Q_0 ((1-A)x, (1-A)x).
\end{eqnarray*}
This shows
$$ Q_0 (Ax,Ax) \leq Q_0 (x,x)$$
for all $x\in D_0$. Now, the implication (iii)$\Longrightarrow$(ii) from  Lemma \ref{Symmetry-Friedrich}  gives (iii).
\end{proof}

\medskip

\textbf{Acknowledgements.}  The authors are grateful to J{\'o}zef Dodziuk for his continued support.  MK and RW would like to thank the Group of Mathematical Physics of the University of Lisbon for their generous backing while parts of this work were completed.  In particular, RW extends his gratitude to Pedro Freitas and Jean-Claude Zambrini for their encouragement and assistance.  RW gratefully acknowledges financial support of the FCT in the forms of grant SFRH/BPD/45419/2008 and project PTDC/MAT/101007/2008.

\end{document}